\documentclass[11pt,a4paper]{article}
\usepackage[utf8]{inputenc}
\usepackage{geometry}
\usepackage{hyperref}
\usepackage{amsmath}
\usepackage{amssymb}
\usepackage{amsthm}
\usepackage{graphicx}
\usepackage{subfig}
\usepackage{svg}
\usepackage{tikz}
\usepackage{tikz-cd}
\usetikzlibrary{intersections,calc}
\usepackage[colorinlistoftodos]{todonotes}
\makeatletter
\DeclareRobustCommand{\cev}[1]{%
	{\mathpalette\do@cev{#1}}%
}
\newcommand{\do@cev}[2]{%
	\vbox{\offinterlineskip
		\sbox\z@{$\m@th#1 x$}%
		\ialign{##\cr
			\hidewidth\reflectbox{$\m@th#1\vec{}\mkern4mu$}\hidewidth\cr
			\noalign{\kern-\ht\z@}
			$\m@th#1#2$\cr
		}%
	}%
}
\makeatother
\usepackage{graphics}
\usepackage{tikz-cd}

\usepackage[mathscr]{euscript}

\theoremstyle{plain}
\newtheorem{Th}{Theorem}[section]
\newtheorem{Lemma}[Th]{Lemma}
\newtheorem{Cor}[Th]{Corollary}
\newtheorem{Prop}[Th]{Proposition}

\theoremstyle{definition}
\newtheorem{Def}[Th]{Definition}

\newtheorem{Rem}[Th]{Remark}
\newtheorem{?}[Th]{Problem}
\newtheorem{Ex}[Th]{Example}

\newcommand{\K}{\mathbb{K}}

\newcommand{\C}{\mathbb{C}}

\newcommand{\tr}{\text{tr}}
\newcommand{\holr}{\text{Hol}^{\text{reg}}}

\DeclareMathOperator{\id}{id}

\DeclareSymbolFont{rsfs}{U}{rsfs}{m}{n}
\DeclareSymbolFontAlphabet{\mathscrsfs}{rsfs}

\newcommand{\ldb}{\{\!\!\{}
\newcommand{\rdb}{\}\!\!\}}

\newcommand{\dl}{d^{\text{L}}}
\newcommand{\dr}{d^{\text{R}}}
\newcommand{\ad}{\mathrm{ad}}

\newcommand{\ft}{\mathfrak{t}}
\newcommand{\ff}{\mathfrak{f}}

\newcommand{\defterm}[1]{\textbf{#1}}

\newcommand{\florian}[2][]{\todo[color=blue!40, #1]{Florian: #2}}

\newcommand{\muze}[2][]{\todo[color=red!40, #1]{Muze: #2}}

\begin{document}

\title{
Poisson brackets and coaction maps of regularized holonomies of the KZ equation
}
\author{Anton Alekseev\thanks{Section of Mathematics, University of Geneva, Rue du Conseil-Général 7-9, 1205 Geneva, Switzerland \href{mailto:anton.alekseev@unige.ch}{anton.alekseev@unige.ch}},~Florian Naef\thanks{School of Mathematics, Trinity College, Dublin 2, Ireland \href{mailto:naeff@tcd.ie}{naeff@tcd.ie}},~Muze Ren\thanks{Section of Mathematics, University of Geneva, Rue du Conseil-Général 7-9, 1205 Geneva, Switzerland \href{mailto:muze.ren@unige.ch}{muze.ren@unige.ch}}}

\date{September 2024}

\maketitle
\begin{abstract}
We derive explicit closed formulas for the Kirillov-Kostant-Souriau (KKS) coaction maps of open path regularized holonomies of the Knizhnik-Zamolodchikov (KZ) equation, and the corresponding Poisson brackets for the Lie algebra ${\rm gl}(N, \mathbb{C})$.
Our main technical tool is a certain projection of the generalized pentagon equation of \cite{AFR2024}.

\end{abstract}

\section{Introduction}

Let $z_1, \dots, z_n$ be distinct points on the complex plane, and let $X_1, \dots, X_n$ be $N$ by $N$ matrices endowed with the linear Kirillov-Kostant-Souriau (KKS) Poisson bracket on their matrix entries. Consider a meromorphic flat connection
$$
\mathcal{A}= \frac{1}{2\pi i} \sum_{i=1}^n X_i \, d\log(z-z_i)  
$$
with $X_i$'s being the residue at the pole located at $z_i$. The corresponding Knizhnik-Zamolodchikov (KZ) equation reads
$$
d\Psi = \mathcal{A}\Psi.
$$
For a closed curve $\gamma \subset \Sigma=\mathbb{C}\backslash \{ z_1, \dots, z_n\}$ define the Goldman function
$$
f_\gamma(X) = {\rm Tr} \, {\rm Hol}(\mathcal{A}, \gamma),
$$
where ${\rm Hol}(A, \gamma)$ is the holonomy of $\mathcal{A}$ along $\gamma$. The function $f_\gamma(X)$ is independent of the choice of the base point on $\gamma$, and it is invariant under homotopy. Furthermore, by the result of Hitchin \cite{Hitchin}, for two closed curves $\gamma_1$ and $\gamma_2$ we have
$$
\{ f_{\gamma_1}(X), f_{\gamma_2}(X)\}_{\rm KKS} = f_{\{\gamma_1, \gamma_2\}_{\rm Goldman}}(X),
$$
where $\{\gamma_1, \gamma_2\}_{\rm Goldman}$ is the Goldman bracket defined on free homotopy classes of loops in $\Sigma$.

One can restate this result using the degree completed free associative algebra
$A=\mathbb{C}\langle\langle x_1, \dots, x_n\rangle\rangle$. The algebra $A$ carries the canonical Hopf algebra structure with the coproduct $\Delta(x_i)=x_i \otimes 1 + 1 \otimes x_i$ and the antipode $S(x_i)=-x_i$. We will be using the Sweedler notation $\Delta(a)=a' \otimes a''$. The quotient
$$
|A| = A/[A,A]
$$
carries a KKS Lie bracket, and the holonomy map is a Lie homomorphism under the Goldman bracket on the source and the KKS Lie bracket on the target (this is a version of the result of \cite{Hitchin}). The first two author have shown \cite{AF2017} that the holonomy map also intertwines the Turaev cobracket (in the blackboard framing) and the KKS (or necklace) cobracket on $|A|$.

In this paper we extend these results to open paths $\gamma$ with tangential endpoints (first introduced in \cite{Deligne1989}), see also of the form $(z_i, 1)$ (here $1$ stands for the tangent vector at the base point $z_i$). In that case, the notion of holonomy has to be replaced by regularized holonomy. In more detail, for a path $\gamma$ starting at the tangential base point $(z_p, 1)$ and ending at the tangential base point $(z_q, 1)$ one defines solutions of the KZ equation defined by their asymptotic behavior:
$$
\Psi_p(z) \sim_{z \to z_p} (z-z_p)^{\frac{x_p}{2\pi i}}, \hskip 0.3cm
\Psi_q(z) \sim_{z \to z_q} (z-z_q)^{\frac{x_q}{2\pi i}}.
$$
The regularized holonomy\footnote{See also \cite{Deligne1989,Deligne1970,Brown2009} for the regularization process.} is defined as the ratio of these two solutions analytically continued along the path $\gamma$:
$$
H={\rm Hol}^{\rm reg}(\mathcal{A}, \gamma)=\Psi_q(z)^{-1} \Psi_p(z).
$$
The path $\gamma$ may have a finite number of transverse self-intersection points labeled by $l=1, \dots, m$ and positioned at
$t_l<s_l$ on the curve. The signs of these self-intersection points (relative to the blackboard orientation of $\mathbb{C}$) are $\varepsilon_l$. The rotation number of $\gamma$ (in the blackboard framing) is denoted by ${\rm rot}(\gamma)$.
Furthermore, we define two functions in one variable:
\begin{equation}  \label{eq:intro_r}
r_\zeta(x):=-\frac{1}{2\pi i}\sum^{\infty}_{n=2} \zeta(n)
\left(\frac{x}{2\pi i}\right)^{n-1}, \hskip 0.3cm
r_{\text{AM}}(x):=-\frac{1}{2}+\sum_{k=1}^{\infty}\frac{1}{(2k)!}B_{2k}x^{2k-1},
\end{equation}
where $\zeta(n)$ are values of the Riemann zeta-function at positive integers, and $B_{2k}$ are Bernoulli numbers. With this notation at hand, we can state our first main result on the reduced KKS coaction map (an upgrade of the Turaev cobracket) of 
$H$:

\begin{Th}\label{Th:intro_1}
We have,
\begin{equation}      \label{eq:intro_theorem1}
\begin{array}{lll}
 \bar{\mu}_\text{KKS}(H)
& =& H r_{\zeta}(-x_p) + {\rm rot}(\gamma) H  - r_{\zeta}(x_q)H
 \\
&+ & \sum_l \varepsilon_l {\rm Hol}(\mathcal{A}, \gamma_{[s_l,1]}) {\rm Hol}(\mathcal{A}, \gamma_{[0,t_l]}) 
- d^{L}_p(H) - d^{R}_q(H),
\end{array}
\end{equation}
where $d^L_i$ and $d^R_i$ are left and right Fox derivatives with respect to $x_i$.
\end{Th}

For more details, see Theorem \ref{Th:reduced KKS formula} in the body of the paper. The explicit and beautiful formula \eqref{eq:intro_theorem1} is the main technical result of the paper. Among other things, it directly applies to the KZ associator \cite{Drinfeld1989,Drinfeld1990} which corresponds to the case of $n=2, z_1=0, z_2=1$, and $\gamma$ the straight path from $z_1$ to $z_2$. Furthermore, equation \eqref{eq:intro_theorem1} implies explicit closed formulas for other operations on regularized holonomies: Fox pairings and double brackets (in the sense of van den Bergh \cite{VandenBergh2008}) for open paths and for closed loops. 

The main tool in proving Theorem \ref{Th:intro_1} is the generalized pentagon equation of \cite{AFR2024}. This equation is set up in the universal enveloping algebra of the Drinfeld-Kohno Lie algebra $\mathfrak{t}_{n,2}$ with generators $t_{i,z}, t_{i,w}$ and $t_{z,w}$ (here labels $z$ and $w$ stand for $n+1$ and $n+2$, respectively). Equation \eqref{eq:intro_theorem1} is obtained by projecting the generalized pentagon equation to the quotient $U\mathfrak{t}_{n,2}/(t_{z,w})^2$, where $(t_{z,w})$ is the two sided ideal spanned by the generator $t_{z,w}$\footnote{Similar constructions appear in forthcoming works of R. Navarro Betancourt and F. Naef \cite{DrorYusuke}, and of D. Bar-Natan and Y. Kuno \cite{RodrigoNaef}}.

The notion of regularized holonomy extends to the original matrix differential equation $d\Psi = \mathcal{A}\Psi$. In that case, regularized holonomies are defined as analytic functions on an open neighborhood of the origin in the space of residues $(X_1, \dots, X_n) \in {\rm Mat}(N, \mathbb{C})^n$. Van den Bergh double brackets give rise to universal formulas for Poisson brackets of matrix entries of regularized holonomies. Our second main result is the explicit formula for the Poisson bracket of matrix entries of regularized holonomies of a pair of loops starting and ending at the same tangential base point:
\begin{Th}\label{Th:intro_2}
For two loops $\gamma_1$ and $\gamma_2$ starting and ending at the tangential base point $(z_m, 1)$ we have the following formula for KKS Poisson brackets of the matrix entries of regularized holonomies $H^{\gamma_1}(X)$ and
$H^{\gamma_2}(X)$:

\begin{align*}
\{H^{\gamma_2}_{ij}(X),H^{\gamma_1}_{uv}(X)\}_{\text{KKS}}=& \sum_{a\in \gamma_1\cap\gamma_2} \varepsilon_a \left(H^{\gamma_1}_{z_m\leftarrow a}H^{\gamma_2}_{a\leftarrow z_m}\right)_{uj}\left( H^{\gamma_2}_{z_m\leftarrow a}H^{\gamma_1}_{a\leftarrow z_m}\right)_{iv} \\
&+\Pi(H^{\gamma_2}_{ij}(X),H^{\gamma_1}_{uv}(X))
\end{align*}
where
\begin{align*}
    \Pi &= \sum_{a,b,c,d} r_{\mathrm AM}(\mathrm{ad}_{X_m})_{ab}^{cd} \, \mathcal{X}^{ab} \otimes \mathcal{X}^{cd} + \sum_{a,b}\frac{\partial}{\partial (X_m)_{ab}} \wedge \mathcal{X}^{ab}
\end{align*}
and $\mathcal{X}$ is the diagonal action of $\mathfrak{gl}_N$ given by
\begin{align*}
    \mathcal{X}^{ab} &= \sum_{l,c} (X_i)_{ac} \frac{\partial}{\partial (X_l)_{cb}} - (X_l)_{cb} \frac{\partial}{\partial (X_l)_{ac}}.
\end{align*}
\end{Th}


For more details, see Theorem \ref{cor:KKS poisson bracketI} in the body of the paper. In the equations above, we separate the contributions of intersection points between $\gamma_1$ and $\gamma_2$ (similar to the Goldman bracket formula), and other contributions collected in the bivector $\Pi$.
Note that the zeta-values $\zeta(n)$ at odd integers do not enter expressions for the Poisson bracket. Only Bernoulli numbers (corresponding to $\zeta(n)$ for $n$ even) are used in the function $r_{\rm AM}(x)$ which enters the answer (see also \cite{AM} for this function).
Theorem \ref{Th:intro_2} was one of the motivations for this work. It is inspired by the results of Korotkin-Samtleben \cite{KS} (in particular, see equation (B.1) in Appendix B) who studied Poisson brackets of holonomies in 1990s.

The structure of the paper is as follows: in Section 2, we recall definitions and main properties of regularized holonomies of the KZ equation. In Section 3, we define and recall the basic properties of the KKS reduced coaction map, the corresonding Fox pairing and the double bracket. In Section 4, we state our main results. Finally, Section 5 is devoted to the proof of Theorem \ref{Th:intro_1}.

\vskip 0.2cm

{\bf Acknowledgements.} We are grateful to D. Bar-Natan, F. Brown, Z. Dancso, B. Enriquez, R. Hain, Y. Kuno, R. Navarro Betancour, and V. Roubtsov for interesting discussions, and for sharing their work with us. Research of AA and MR was supported in part by the grants 08235 and 20040, and by the National Center for Competence in Research SwissMAP of the Swiss National Science Foundation. Research of A.A. was supported in part by the award of the Simons
Foundation to the Hamilton Mathematics Institute of the Trinity College Dublin
under the program ``Targeted Grants to Institute''.

\section{Generalized pentagon equation}  \label{sec:pentagon}

In this Section, we briefly recall the notion of regularized holonomy for the KZ equation, and we state the generalized pentagon equation of \cite{AFR2024}.

\subsection{KZ connection}\label{sub:local_KZ}

Let $\mathfrak{t}_{n+1}$ be the Drinfeld-Kohno Lie algebra on $n+1$ strands. 
We denote its generators by $t_{i,j}$ for $i,j \in \{ 1, \dots, n, z \}$, where $z$ stands for $n+1$, and the relations are given by
\begin{equation} \label{eq:DK}
[t_{i,j}, t_{k,l}]=0, \hskip 0.3cm [t_{i,j} + t_{i,k}, t_{j,k}]=0
\end{equation}
for distinct indices $i,j,k,l$. Recall that the generators
$t_{i,z}$ for $i=1, \dots, n$ span a free Lie subalgebra $\mathfrak{f}_n \subset \mathfrak{t}_{n+1}$.

\begin{equation} \label{eq:DK}
[t_{i,j}, t_{k,l}]=0, \hskip 0.3cm [t_{i,jk}, t_{j,k}]=0
\end{equation}
for distinct indices $i,j,k,l$. Here we use the notation $t_{i,jk}=t_{i,j} + t_{i,k}$.
For the surface $\Sigma = \mathbb{C}\backslash \{ z_1, \dots, z_n \}$, we consider 
the following (part of) the KZ connection:
\begin{equation}
\mathcal{A}=\frac{1}{2\pi i}\sum_{i=1}^{n} t_{i,z}d\log(z-z_i).
\end{equation}
It defines the differential equation 
\begin{equation}\label{eq:KZ}
d\Psi=\mathcal{A}\Psi,
\end{equation}
where $\Psi(z) \in U\mathfrak{f}_n \cong \mathbb{C}\langle\langle t_{1,z}, \dots, t_{n,z}\rangle\rangle$ take values in the (completion of) the universal enveloping algebra of $\mathfrak{f}_n$.

\begin{Rem}
Note that the symbol $z$ denotes both the last strand and the (coordinate of the) point moving on $\Sigma$.
\end{Rem}

In order to state the generalized pentagon equation, we will also need the Lie algebra 
$\mathfrak{t}_{n,2} \subset \mathfrak{t}_{n+2}$ with the label $n+2$ denoted by $w$, and with generators
$t_{i,z}, t_{i,w}$ for $i=1, \dots, n$ and $t_{z,w}$. Denote by $(t_{z,w})$ the two sided ideal spanned by $t_{z,w}$ and recall that 
$$
\pi_0: Ut_{n,2} \to Ut_{n,2}/(t_{z,w}) \cong U\mathfrak{f}_n \otimes U\mathfrak{f}_n,
$$
where $\pi_0$ is the canonical porjection, and the two copies of $U\mathfrak{f}_n$ on the right hand side are spanned by $t_{i,z}$ and $t_{i,w}$. 
For a pair of labels $p, q$, we introduce the following versions of the KZ connections:
\begin{subequations}\label{eq:intro_connections}
\begin{equation}
\mathcal{A}_z:= \frac{t_{z,wq}}{2\pi i}\, d\log(z-z_q)+\sum_{k\ne q}\frac{t_{k,z}}{2\pi i}\, d\log(z-z_k),   
\end{equation}
\begin{equation}
\mathcal{A}_w:= \frac{t_{pz,w}}{2\pi i} \, d\log(w-z_p)+\sum_{k\ne p }\frac{t_{k,w}}{2\pi i}\, d\log(w-z_k),
\end{equation}
\begin{equation}
\mathcal{A}_{zw}:= \sum_{k=1}^n \frac{t_{k,zw}}{2\pi i} \, d\log(z-z_l).
\end{equation}
\end{subequations}

\subsection{Regularized holonomy}       \label{sub:local_solutions}
Let $\gamma: [0,1] \to \mathbb{C}$ be a smooth path with $\dot{\gamma}= d\gamma(t)/dt \neq 0$ which may start and end at regular or marked points on the complex plane, but with $\gamma(t) \in \Sigma =\mathbb{C}\backslash \{ z_1, \dots, z_n\}$ for all $0<t<1$.
In case when $\gamma$ starts or ends at marked points (or both), we require that near that point $\gamma(t)$ be linear in $t$:
$$
\begin{array}{ll}
\gamma(t) = z_p + v_p t & {\rm for} \,\, 0< t < \epsilon; \\
\gamma(t) = z_q + v_q(1-t) & {\rm for} \,\, 0< 1-t < \epsilon.
\end{array}
$$
Here $(z_p, v_p)$ is the tangential starting point of $\gamma$, $(z_q, v_q)$ is the tangential end point of $\gamma$, and $\epsilon$ is small enough. The path $\gamma(t)$ may have a finite number of transverse self-intersection points, $\gamma(s_l) = \gamma(t_l) = a_l$ for $l=1, \dots, m$, with intersection signs $\epsilon_l$. In that case, we also require that $\gamma(t)$ be a linear function near the points $s_l$ and $t_l$.
To each path $\gamma(t)$, we associate two local solutions $\Psi_p(z)$ and $\Psi_q(z)$ corresponding to its end points (regular or tangential). 

\begin{Def}
For a path $\gamma$, the (regularized) holonomy of $\mathcal{A}_z$ along $\gamma$
is given by
\begin{equation}\label{eq:def_regularized_hol}
\holr(\mathcal{A}_z,\gamma):=\Psi_q^{-1}(z)\Psi_p(z)
\in U\mathfrak{f}_n,
\end{equation}
where $z$ is any point in the strip around $\gamma$. 
\end{Def}

For regular end points, this definition coincides with the standard definition of a parallel transport, and the corresponding holonomy is invariant under homotopies which preserve end points. Similarly, for tangential end points the regularized holonomy is invariant under homotopies which preserve $(\gamma(0), \dot{\gamma}(0))$ for the starting point and $(\gamma(1), -\dot{\gamma}(1))$ for the end point.
For a path $\gamma$ starting at $(z_p, 1)$ and ending at $(z_q, 1)$, we define the regularized holonomies
\begin{equation}      \label{eq:intro_honolomies}
H_z={\rm Hol}^{\rm reg}(\mathcal{A}_z, \gamma), \hskip 0.3cm
H_w={\rm Hol}^{\rm reg}(\mathcal{A}_w, \gamma), \hskip 0.3cm
H_{zw}={\rm Hol}^{\rm reg}(\mathcal{A}_{zw}, \gamma).
\end{equation}

If both end points of the  path $\gamma$ are tangential, one can define its rotation number with respect to the blackboard framing:
\begin{equation}\label{eq:rotation_number}
\rm{rot}_{\gamma}=\frac{1}{2\pi}\, {\rm Im} \, \int_{0}^{1}d\log(\dot{\gamma}(t)).
\end{equation}
We have ${\rm rot}_\gamma \in \mathbb{R}$ and one full turn in the anti-clockwise direction corresponding to ${\rm rot}_\gamma=1$. Note that one can consistently choose $v=1$ in all tangetial base points $(z_i, 1)$.  In that case, all rotation numbers of paths are half-integers:
${\rm rot}_\gamma \in \mathbb{Z} + 1/2$. Furthermore,  the regular fundamental group $\pi^{\rm reg}_1(\Sigma, (z_i, v_i))$ surjects onto the ordinary fundamental group $\pi_1(\Sigma, z_i)$.

We can now state the generalized pentagon equation:
\begin{Th}[\cite{AFR2024}]    \label{thm:generalized_pentagon_equation}
	There exist elements $C_l\in \exp(\mathfrak{t}_{n,2})$ for $l=1, \dots, m$ such that the following identity holds in $U(\mathfrak{t}_{n,2})$,

 \begin{equation} \label{eq:general_pentagon}
\Phi_{\rm{KZ}}(t_{zw},t_{wq})(H_{zw} e^{\text{rot}(\gamma)t_{z,w}})\Phi_{\rm{KZ}}(t_{pz},t_{zw})=H_z \left(\prod_{l=1}^{m}C_{l}^{-1}e^{-\epsilon_l t_{z,w}}C_{l}\right) H_w.
 \end{equation}	
    Furthermore,
    \begin{equation}
        \pi_0(C_l) = {\rm Hol}^{\rm reg}(\pi_0(\mathcal{A}_z), \gamma_{[0,t_l]}) 
        {\rm Hol}^{\rm reg}(\pi_0(\mathcal{A}_w), \gamma_{[1,s_l]}).
    \end{equation}
\end{Th}

\subsection{Composition of paths} \label{sub:composition}

Two paths $\gamma_1$ and $\gamma_2$ with end points (regular or tangential) $(p_1, q_1)$ and $(p_2, q_2)$ are called composable if $p_2 = q_1$. In that case, we can define a composition of paths $\gamma_2 \gamma_1$. 
If the common end point $p_2 = q_1$ is a regular point, the composition is defined up to homotopy by concatenation of paths (this allows to make the resulting path smooth). If the common end point is a marked point, the composition is defined modulo regular homotopy (that is, the first Reidemeister move is not allowed). In that case, there are two natural ways to define the composition. One considers a neighborhood of the marked point (see Fig. \ref{fig:composing3}), and adds a clockwise or a counter-clockwise half-turn with respect to the blackboard orientation of $\mathbb{C}$ (see Fig. \ref{fig:composing}). In this paper, we chose to use the clockwise convention.

\begin{figure}[h]
    \centering    \includegraphics[width=0.6\columnwidth]{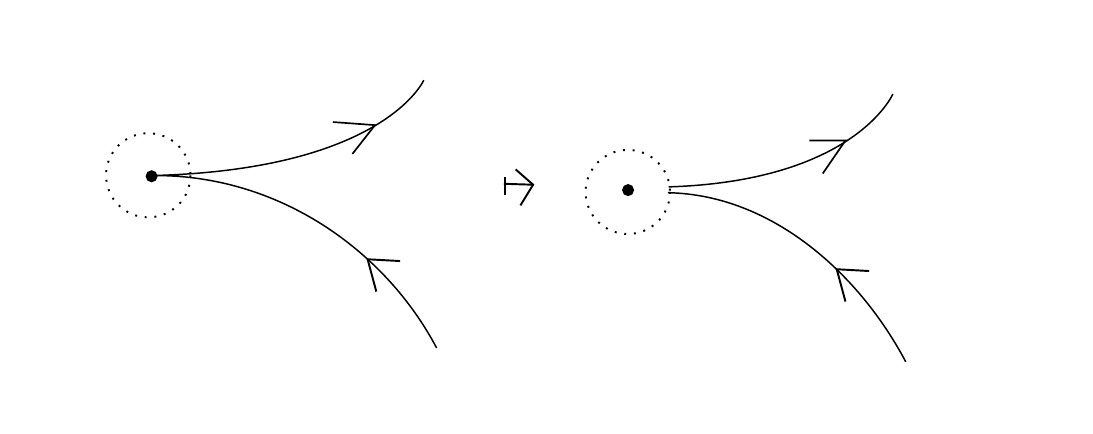}
    \caption{small region near the marked point}
    \label{fig:composing3}
\end{figure}

\begin{figure}[h]
    \centering   \includegraphics[width=0.6\columnwidth]{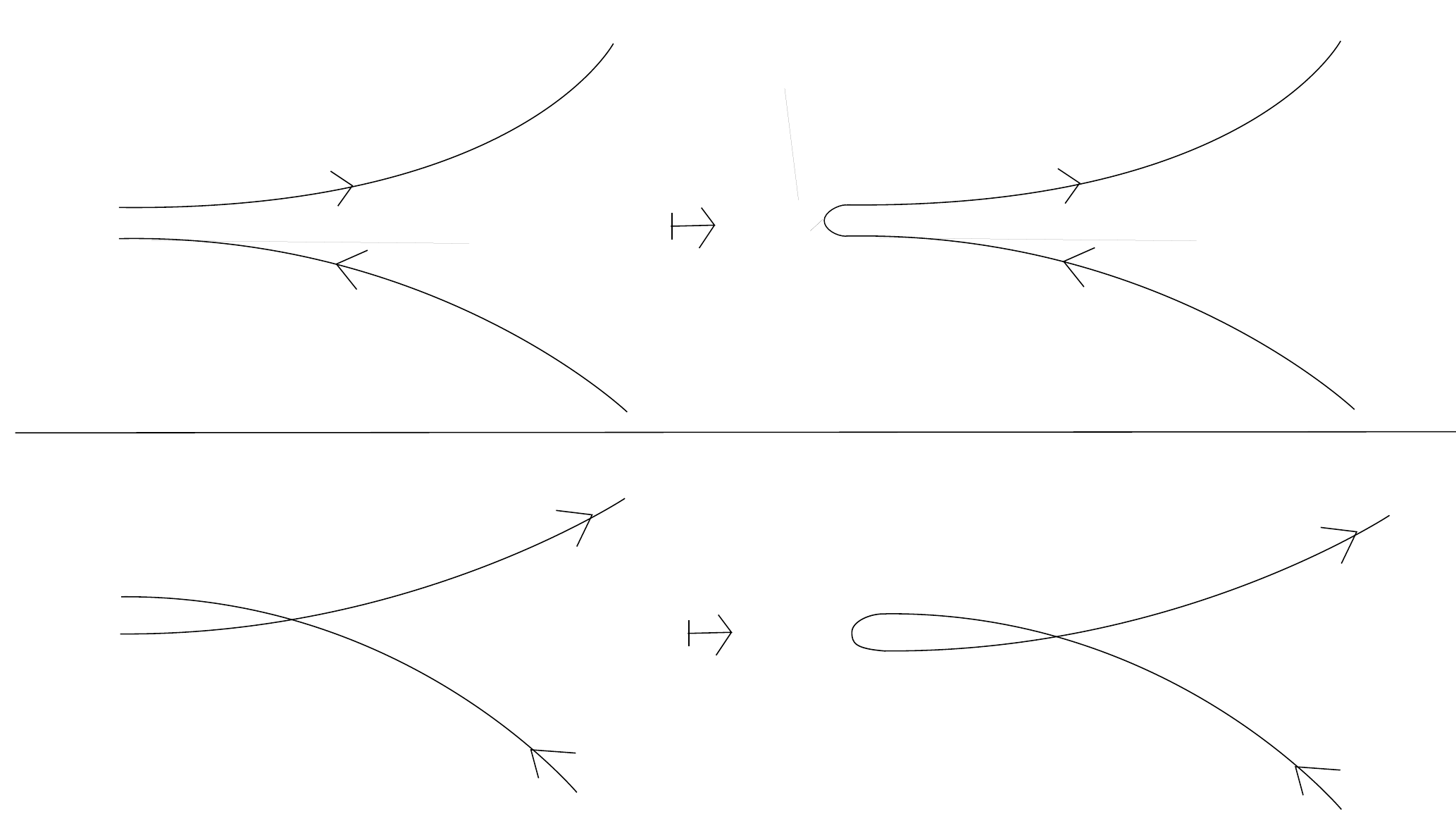}
    \caption{two ways of composing curves}
    \label{fig:composing}
\end{figure}

Recall that holonomies and regularized holonomies are multiplicative under composition of paths:
\begin{equation}\label{eq:composable}
{\rm Hol}^{\rm reg}(\mathcal{A}, \gamma_2 \gamma_1) = \holr(\mathcal{A}, \gamma_2) \holr(\mathcal{A}, \gamma_1).
\end{equation}
For the composition of paths with tangential endpoints (hence, defined up to regular homotopy), we have
$$
{\rm rot}_{\gamma_2 \gamma_1} =
{\rm rot}_{\gamma_1} + {\rm rot}_{\gamma_2}- \frac{1}{2},
$$
where $-1/2$ on the right hand side comes from the extra clockwise half-turn in the definition of the path composition $\gamma_2\gamma_1$.

\section{Hopf algebras with double brackets and coaction maps}
In this section, we recall definitions and basic properties of Fox pairings and their relation to double brackets \cite{VandenBergh2008} following \cite{MassuyeauTuraev2014}), and coaction maps following \cite{KK2014,Massuyeau}.

\subsection{Double bracket and Fox pairing}
Let $(A, \Delta, S, \varepsilon)$ be an involutive Hopf algebra with coproduct $\Delta$, antipode $S$, and counit $\varepsilon$ over a field $\mathbb{K}$ of characteristic zero.  Recall that a Hopf algebra is involutive if the antipode squares to identity.

Our main example is the degree completed free associative algebra $A = \mathbb{K}\langle\langle x_1, \dots, x_n \rangle\rangle$ with generators $x_1, \dots, x_n$. It is equipped with the coproduct
$$
\Delta(x_i) = x_i \otimes 1 + 1 \otimes x_i,
$$
the antipode $S(x_i)=-x_i$,  and the counit $\epsilon(x_i) =0$. In this caee, it is easy to check the involutivity property. In what follows, all tensor products will be understood as degree completed tensor products.

We recall the definitions of Fox pairing and of Fox derivative following \cite{MassuyeauTuraev2014}. 
\begin{Def}[Fox derivative]
A $\K$-linear map $\partial: A\to A$ is called a \defterm{left Fox derivative} if
\begin{equation}
\partial(ab)=a\partial(b)+\partial(a)\varepsilon(b).
\end{equation}
A map $\partial: A \to A$ is a \defterm{right Fox derivative} if
\begin{equation}
\partial(ab)=\partial(a)b+\varepsilon(a)\partial(b)
\end{equation}
for all $a,b\in A$.
\end{Def}

\begin{Ex}\label{example:left_right fox deri}
For any $x\in A$, there are unique presentations
\begin{equation}
x=\varepsilon(x)+\sum^n_{i=1}x_id^R_i(x)=\varepsilon(x)+\sum^n_{i=1}d^{L}_i(x)x_i
\end{equation}
which define the right Fox derivatives $d^{R}_i$ and the left Fox derivatives $d^{L}_i$.
\end{Ex}

\begin{Def}[Fox pairing]
\phantom{ }
\begin{itemize}
    \item 	A bilinear map $\rho:A\otimes A\to A$ is a \defterm{Fox pairing} if it is a left Fox derivative in its first argument, and it is a right Fox derivative in its second argument. That is,
	\begin{equation}
	\rho(ab, c)=a\rho(b, c)+\rho(a, c)\varepsilon(b), \hskip 0.3cm \rho(a,bc)=\rho(a,b)c+\varepsilon(b)\rho(a,c)
	\end{equation}
	for all $a,b,c \in A$.
    \item The \defterm{transpose} of a Fox pairing $\rho$ is the bilinear form $\bar{\rho}:A\times A\to A$ defined by
\begin{equation}
\bar{\rho}(a,b)=S(\rho(S(b),S(a))).
\end{equation}
    \item The Fox pairing $\rho$ is called \defterm{skew-symmetric} if
    \[
    \bar{\rho}=-\rho.
    \]
    \item To each $g \in A$ one can associate an \defterm{inner} Fox pairing
    \begin{equation}
    \rho_g(a,b)=(a-\varepsilon(a))g(b-\varepsilon(b)).
    \end{equation}
\end{itemize}

\end{Def}

\begin{Ex}
On the free algebra 
$A=\K\langle\langle x_1,\dots,x_n\rangle\rangle$, we define
\begin{align*}
	\rho_{\rm{KKS}}(x_i,x_j):=\delta_{ij}x_i,
\end{align*}
It uniquely extends to a Fox pairing $\rho_{\text{KKS}}$ on $A$ by $\rho_{\rm{KKS}}(x,1)=\rho_{\rm{KKS}}(1,x)=0$ for all $x\in A$. In particular, on monomials $h_1\dots h_m, k_1\dots k_n\in A$ we have 
\begin{align*}
	\rho_{\rm{KKS}}(h_1\dots h_m,k_1\dots k_n):=h_1\dots h_{m-1}(\rho_{\rm{KKS}}(h_m,k_1))k_2\dots k_n.
\end{align*} 
\end{Ex}

\begin{Prop}
    Let $\rho \colon A \otimes A \to A$ be a Fox pairing. Then,
    \[
    c(a_1 \otimes b_1, a_2 \otimes b_2) := \varepsilon(a_1) \rho(b_1,a_2) \varepsilon(b_2),
    \]
    defines a $2$-cocycle $(A \otimes A)^{\otimes 2} \to M$, where 
    $M \cong A$ endowed with the $A \otimes A$-bimodule structure given by
    \[
    (f \otimes g) a (h \otimes k) = \varepsilon(f)\varepsilon(k) \, gah.
    \]
    Furthermore, the formula
    \begin{equation}
    (a_1 \otimes b_1 \oplus c_1) \cdot_{\rho} (a_2 \otimes b_2 \oplus c_2) := a_1 a_2 \otimes b_1 b_2 \oplus \left( \varepsilon(a_1) \, b_1 c_2 + c_1 a_2 \, \varepsilon(b_2) + \varepsilon(a_1) \rho(b_1,a_2) \varepsilon(b_2) \right)
    \label{eqn:multiplicationfromfoxp}
    \end{equation}
    defines an algebra structure on $A \otimes A \oplus M$.
\label{prop:forpairing2cocycle}
\end{Prop}
\begin{proof}
For every triple $(a_1,b_1),(a_2,b_2),(a_3,b_3) \in A \otimes A$, one verifies the cocycle property of $c$
\begin{equation*}
(a_1\otimes b_1)c(a_2\otimes b_2, a_3\otimes b_3)+c(a_1\otimes b_1,a_2a_3\otimes b_2b_3)=c(a_1a_2\otimes b_1b_2,a_3\otimes b_3)+c(a_1\otimes b_1,a_2\otimes b_2)(a_3\otimes b_3)
\end{equation*}
by a direct calculation.
\end{proof}

Note that $M$ can also be described as the $A \otimes A$-bimodule obtained by restriction along the maps
$\id \otimes \varepsilon$ on the left and $\varepsilon \otimes \id$ on the right from the standard $A$-bimodule, where $A$ acts on itself by the left and right muptiplication.
By abuse of notation, we sometimes identify $M$ with $A$ so that we can write
\[
amb=(1 \otimes a)m(b \otimes 1)
\]
for all $a, b \in A, m \in M$.


\begin{Def}[Double bracket]
A linear map $\ldb-,-\rdb: A\otimes A\to A\otimes A, \, a\otimes b\mapsto \ldb a,b\rdb$ is a \defterm{double bracket} if 
\begin{subequations}
\begin{equation}
\ldb a,bc\rdb=(b\otimes 1)\ldb a,c\rdb+ \ldb a,b\rdb (1\otimes c)
\end{equation}
\begin{equation}
\ldb b,a\rdb=-P_{21}\ldb a,b\rdb,
\end{equation}
\end{subequations}
where $P_{21}$ is the permutation defined by $P_{21}(a\otimes b) = b \otimes a$. These two properties imply that
\begin{equation}
\ldb ab,c\rdb=\ldb a,c\rdb (b \otimes 1) + (1 \otimes a) \ldb b,c\rdb
\end{equation}
\end{Def}
The following theorem establishes a link between Fox pairings and double brackets.

\begin{Th}[\cite{MassuyeauTuraev2014}, Lemma 6.2\label{th:Fox_to_Double}]
	Let $(A, \Delta, S, \epsilon)$ be an involutive Hopf algebra, and $\rho: A\times A\to A$ be a skew-symmetric Fox pairing on $A$. Then, the  map $\ldb-,-\rdb^{\rho}:A\otimes A\to A\otimes A$ defined by formula
	\begin{equation}\label{eq:pho to double}
		\ldb a,b\rdb^{\rho}:=
		b'S(\rho(a'',b'')')a'\otimes \rho(a'',b'')''
	\end{equation}
is a double bracket.
\end{Th}

\begin{Ex}[KKS double bracket]       \label{ex:KKSdoublebracket}
For $A=\mathbb{K}\langle\langle x_1, \dots, x_n\rangle\rangle$ with the Fox pairing $\rho_{\text{KKS}}$, we compute the corresponding double bracket:
\begin{equation}
	\ldb x_i,x_j\rdb_{\rm KKS}=S(\rho_{\rm{KKS}}(x_i, x_j)) \otimes 1 + 1 \otimes \rho_{\rm{KKS}}(x_i, x_j) = \delta_{ij}(1 \otimes x_i - x_i \otimes 1),
\end{equation}
here we have used that $\rho_{\rm{KKS}}(1, a)=\rho_{\rm{KKS}}(a, 1)=0$ for all $a \in A$.
\end{Ex}

In the opposite direction, if the double bracket is constructed using a Fox pairing, the Fox pairing can be recovered from the double bracket. Indeed,
$$
\begin{array}{lll}
(\varepsilon \otimes 1)\ldb a, b\rdb & = &
\varepsilon(b'S(\rho(a'',b'')')a')\, \rho(a'',b'')'' \\
& = & \varepsilon(b') \varepsilon(S(\rho(a'',b'')')) \varepsilon(a')\, \rho(a'',b'')'' \\
& = & \rho(\varepsilon(a')a'',\varepsilon(b')b'') \\
& = & \rho(a,b).
\end{array}
$$

Let $A$ be an algebra, we define $|A| = A/[A, A]$, where $[A, A]$ is the vector space spanned by commutators, and $|\cdot|$ the projection map $A\to |A|$.

\begin{Lemma}[\cite{VandenBergh2008}, Section 6]
The double bracket $\ldb-,-\rdb_{\text{KKS}}$ induces a Lie bracket on $|A|=A/[A,A]$ defined by
\begin{equation}       \label{eq:KKSbracket}
\{|a|,|b|\}_{\rm necklace}=|\ldb a,b\rdb'_{\text{KKS}}\ldb a,b\rdb''_{\text{KKS}}|.
\end{equation}
\end{Lemma}
\begin{Rem}
The Lie bracket \eqref{eq:KKSbracket} coincides with the necklace Lie bracket associated to the star shape quiver with the central vertex connected to $n$ outside vertices.
\end{Rem}

In the follows, we show that certain Fox pairings induce vanishing brackets on $|A|$. By abuse the notation, we also define:
\begin{equation}
|\quad |: \ldb a,b\rdb\mapsto |\ldb a,b \rdb'\ldb a,b \rdb''|
\end{equation}

\begin{Prop}[\cite{MassuyeauTuraev2014}, Lemma 6.3]\label{prop: vanish_inner}
If $A$ is involutive, then for any $a,b,g\in A$, the double bracket induced by the inner Fox pairing $\rho_g$ is
\begin{equation}
\ldb a,b\rdb^{\rho_{g}}=\sum_{(g)}(S(g')\otimes ag''b+bS(g')a\otimes g''-bS(g')\otimes ag''-S(g')a\otimes g''b)
\end{equation}
which implies $|\ldb a,b\rdb^{\rho_{g}}|=0$.
\end{Prop}

\begin{Prop}\label{prop:vanish_partial_inner}
For any $m = 1, \dots, n$, the operations 
\begin{equation}
\rho_{L,m}(a,b):=d^{L}_m(a)(b-\varepsilon(b))
\end{equation}
and
\begin{equation}
\rho_{R.m}(a,b):=(a-\varepsilon(a))d^R_{m}(b)
\end{equation}
are Fox pairings and $|\ldb a,b\rdb^{\rho_{L,m}}|=|\ldb a,b\rdb^{\rho_{R,m}}|=0$.
\end{Prop}

\begin{proof}
By direct computation we have
\begin{equation*}
\ldb a,b\rdb^{\rho_{L,m}}=S(d^L_m(a'')')a'\otimes d^L_m(a'')''b-bS(d^L_m(a'')')\otimes d^L_m(a'')''
\end{equation*}	
\begin{equation*}
\ldb a,b\rdb^{\rho_{R,m}}=b'S(d^R_m(b'')')\otimes ad^R_m(b'')''-b'S(d^R_m(b'')')a\otimes d^R_m(b'')''
\end{equation*}
which implies 
$|\ldb a,b\rdb^{\rho_{L,m}}|=|\ldb a,b\rdb^{\rho_{R,m}}|=0$.
\end{proof}

\subsection{Coaction maps}

\begin{Def}
A map $\mu: A \to |A| \otimes A$ is called a \defterm{coaction} map with respect to the double bracket $\ldb \cdot, \cdot \rdb$ if it satisfies the following product rule
\begin{equation}   \label{eq:mu_mult}
	\mu(ab)=\mu(a)(1\otimes b)+(1\otimes a)\mu(b)+(|-|\otimes 1)\ldb a,b\rdb.
\end{equation}
we also define the \defterm{reduced coaction} map $\bar{\mu}$ associated to $\mu$.
$$
\bar{\mu}=(\varepsilon \otimes 1)\mu: A \to A.
$$
\end{Def}

\begin{Ex}\label{example_mu}
The Hopf algebra $A=\mathbb{K}\langle\langle x_1, \dots, x_n\rangle\rangle$ with the KKS double bracket is equipped with a coaction map defined on generators by $\bar{\mu}_{\rm{KKS}}(x_i)=0$, and by 
$\bar{\mu}_{\rm{KKS}}(1)=0$. On monomials $h=h_1\dots h_m\in A$, $h_i\in \{x_j,1\le j\le n\}$, the map $\bar{\mu}_{\rm KKS}$ is given by
\begin{equation}
\bar{\mu}_{\rm{KKS}}(h)=\sum_{j<k}\delta_{h_j,h_k} (|h_{j+1}\dots h_{k-1}|\otimes h_1\dots h_jh_{k+1}\dots h_m-|h_j\dots h_{k-1}|\otimes h_1\dots h_{j-1}h_{k+1}\dots h_m)
\end{equation}
and the reduced coaction is:
\begin{equation}
\bar{\mu}_{\rm{KKS}}(h_1\dots h_m)=\sum_{1\le i\le m-1} \delta_{h_i,h_{i+1}}
h_1\dots h_ih_{i+2}\dots h_m
\end{equation}
\end{Ex}
\begin{Prop}
Let $(A,\Delta,S,\varepsilon)$ be an involutive Hopf algebra equipped with a skew-symmetric Fox pairing $\rho$. Suppose that the linear map $\bar{\mu}:A\to A$ satisfies the property 
\begin{equation}
\bar{\mu}(ab)=\bar{\mu}(a)b+a\bar{\mu}(b)+\rho(a,b).
\label{eqn:quasi-derivation}
\end{equation}
Then, the map from $A\to |A|\otimes A$ defined by 
$$
\mu(a)=|a'S(\bar{\mu}(a'')')|\otimes \bar{\mu}(a'')''
$$ 
is a coaction map with respect to the double bracket $\ldb-,-\rdb^{\rho}$. 
\end{Prop}

\begin{proof}
It is convenient to introduce the notation
\begin{equation}
d_{\bar{\mu}}(a)=a'S(\bar{\mu}(a'')')\otimes \bar{\mu}(a'')''.
\end{equation}
Using the Sweedler notation
\begin{equation}
(\Delta\otimes \id)\Delta(a)=(\id\otimes \Delta)\Delta(a)=\sum a'\otimes a''\otimes a''',
\end{equation}
one can write the product formula for $d_{\bar{\mu}}$:
\begin{equation}
\begin{split}
d_{\bar{\mu}}(ab)&=a'b'S(\bar{\mu}(b'')')S(a'')\otimes a''' \bar{\mu}(b'')''+a'S(\bar{\mu}(a'')')\otimes \bar{\mu}(a'')''b\\
&+a'b'S(\rho(a'',b'')')\otimes \rho(a'',b'')''
\end{split}
\end{equation}
which implies the product rule \eqref{eq:mu_mult}.
\end{proof}

\begin{Rem}
Linear maps $\bar{\mu} \colon A \to A$ satisfying equation \eqref{eqn:quasi-derivation} are introduced and studied in \cite{Massuyeau} under the name quasi-derivation. The fomula for $d_{\bar{\mu}}$ is also introduced there.
\end{Rem}

\begin{Prop}\label{prop:relation_KKS}
Let $(A,\varepsilon)$ be an augmented algebra equipped with a double bracket $\ldb-,-\rdb$, the Fox pairing $\rho=(\varepsilon\otimes 1)\ldb-,-\rdb$, and the coaction map $\mu$. Then, we have the following relation  
$$
\rho(a, b)= \bar{\mu}(ab)-a\bar{\mu}(b) - \bar{\mu}(a)b.
$$
\end{Prop}

\begin{proof}
By applying $(\varepsilon \otimes 1)$ to equation \eqref{eq:mu_mult}, we obtain
$$
\bar{\mu}(ab)=\bar{\mu}(a)b+a\bar{\mu}(b)+(\varepsilon \otimes 1)\ldb a,b\rdb = \bar{\mu}(a)b+a\bar{\mu}(b) + \rho(a,b),
$$
as required.
\end{proof}

\begin{Ex}
For the KKS bracket and the corresponding coaction map, we compute
$$
(\varepsilon \otimes 1)\mu_{\rm{KKS}}(x_jx_i) =
(\varepsilon \otimes 1)\ldb x_j, x_i \rdb_{\rm KKS} =
\delta_{ij}(\varepsilon \otimes 1)(1 \otimes x_i - x_i \otimes 1) =
\delta_{ij} x_i,
$$
as expected.
\end{Ex}

The reduced coaction is also closely related to the necklace cobracket introduced in \cite{Schedler}.

\begin{Prop}[\cite{Massuyeau}, equation (8.26) ]\label{prop:necklace cobracket}
The reduced coaction $\bar{\mu}_{\text{KKS}}$ induces a cobracket on the space $|A|$,
\begin{equation}\label{eq:delta}
\delta_{\rm necklace}=|d_{\bar{\mu}_{\text{KKS}}}|-|P_{21}d_{\bar{\mu}_{\text{KKS}}}|
\end{equation}
where $P_{21}$ is the permutation. The cobracket $\delta_{\rm necklace}$ coincides with the necklace cobracket associated to the star shape quiver with one central vertex connected to $n$ outside vertices in \cite{Schedler}.
\end{Prop}

\begin{Rem}
In the setup of equation (8.26) \cite{Massuyeau}, the cobracket is defined in the quotient $|A|/\K 1$. And equation  \eqref{eq:delta} corresponds to the $n=0$ case of $\xi$ in \cite{Massuyeau} which is true on $|A|$, $p$ in $\cite{Massuyeau}$ is the identity for us.  
\end{Rem}

\subsection{Double derivations}

\begin{Def}
	A double derivation on $A$ is a $\K$ linear map $\phi:A\to A\otimes A$ satisfying
	\begin{equation}
		\phi(ab)=\phi(a)(1\otimes b)+(a\otimes 1)\phi(b)
	\end{equation}
\end{Def}

\begin{Rem}
	The definition above uses the standard $A\otimes A$-bimodule structure on $A\otimes A$.
 In fact, this definition can be used for any $A\otimes A$-bimodule structure on $A\otimes A$.
\end{Rem}

We now define two linear endomorphisms of $A\otimes A$ in order to compare different bimodule structures.
\begin{equation}
\begin{array}{ll}
\alpha:A\otimes A\mapsto A\otimes A & a\otimes b\mapsto a S(b')\otimes b'', \\
\beta: A\otimes A\mapsto A\otimes A & a\otimes b\mapsto b'\otimes S(b'')a.
\end{array}
\end{equation} 

\begin{Lemma}
$\alpha$ and $\beta$ are invertible with the following inverse maps
\begin{equation}
\begin{array}{ll}
\alpha^{-1}:A\otimes A\mapsto A\otimes A & c\otimes d\mapsto cd'\otimes d'', \\
\beta^{-1}:A\otimes A\mapsto A\otimes A & c\otimes d\mapsto c''d\otimes c'.
\end{array}
\end{equation}
\end{Lemma}

\begin{proof}
The proof is by a direct calculation.
\end{proof}

With the help of the map $\alpha$ and $\beta$, we could build the following relationship between Fox derivatives and double derivations.
\begin{Prop}\label{prop:fox to double derivation}
Suppose that $\delta_{R}$ is a right Fox derivative and $\delta_{L}$ is a left Fox derivative, then the maps
$$
\begin{array}{l}
D\delta_{R}:=\alpha(\id\otimes \delta_R)\Delta: A\mapsto A\otimes A, \\
D\delta_{L}:=\beta(\id\otimes \delta_L)\Delta: A\mapsto A\otimes A
\end{array}
$$
are double derivations.
\end{Prop}
\begin{proof}
We directly check the derivation property
\begin{align*}
D\delta_{R}(ab)&=a'\epsilon(a'')b'S(\delta_{R}(b'')')\otimes \delta_{R}(b'')''+a'b'S((b'')')\otimes \delta_{R}(a'')''(b'')''\\
&=(a\otimes 1)D\delta_{R}(b)+D\delta_{R}(a)(1\otimes b),
\end{align*}
\begin{align*}
D\delta_{L}(ab)&=(a'')'\delta_{L}(b'')'\otimes S(\delta_{L}(b'')'')S((a'')'')a'b'+\delta_{L}(a'')'\otimes S(\delta_{L}(a'')'')a'b\\
&=(a\otimes 1)D\delta_{L}(b)+D\delta_{L}(a)(1\otimes b).
\end{align*}
\end{proof}

In what follows we will need the following double derivations constructed from the left Fox derivative $d^L_m$ and the right Fox derivative $d^R_m$ by the Proposition \ref{prop:fox to double derivation}.
\begin{Cor}\label{prop: left_right_contribution}
\begin{equation*}
Dd^R_m=Dd^L_m
\end{equation*}
\end{Cor}
\begin{proof}
It is now sufficient to prove the equality between the two double derivations on generators:
$Dd^R_m(x_n)=Dd^L_m(x_n)=\delta_{m,n}(1\otimes 1)$.
\end{proof}

\section{Main results}

In this Section, we state the results for the KKS coaction map, the Fox pairing, and the KKS double brackets of regularized holonomies with tangential base points.  We also give an application of these results to Poisson geometry. In what follows, we identify $\C\langle\langle t_{1,z}, \dots, t_{n,z}\rangle\rangle$ of Section 2 with $\C\langle\langle x_{1}, \dots, x_{n}\rangle\rangle$ of Section 3.

\subsection{KKS operations for KZ holonomies}

Recall the functions in one variable defined in \eqref{eq:intro_r}:
\[
r_\zeta(x):=-\frac{1}{2\pi i}\sum^{\infty}_{n=2} \zeta(n)
\left(\frac{x}{2\pi i}\right)^{n-1}, \hskip 0.3cm
r_{\text{AM}}(x):=-\frac{1}{2}+\sum_{k=1}^{\infty}\frac{1}{(2k)!}B_{2k}x^{2k-1}.
\]
Note that
$r_{\text{AM}}(x)=r_{\zeta}(x)-r_{\zeta}(-x)-\frac{1}{2}$.
Our first main result is the following theorem:

\begin{Th}\label{Th:reduced KKS formula}
Let $\gamma$ be a path starting at the tangential base point $(z_p, 1)$ and ending at the tangential base point $(z_q, 1)$, $H^\gamma ={\rm Hol}^{\rm reg}(\mathcal{A}, \gamma) $, and $l=1, \dots, m$ label self-intersection points of $\gamma$, $t_l<s_l$ label the position of self intersection points on the curve. Then, we have

\begin{equation}      \label{eq:mubar}
\begin{array}{lll}
 \bar{\mu}_\text{KKS}(H^\gamma)
& =& H^\gamma r_{\zeta}(-x_p) + {\rm rot}(\gamma) H^\gamma  - r_{\zeta}(x_q)H^\gamma
 \\
&+ & \sum_l \epsilon_l {\rm Hol}^{\rm reg}(\mathcal{A}, \gamma_{[s_l,1]}) {\rm Hol}^{\rm reg}(\mathcal{A}, \gamma_{[0,t_l]}) 
- d^{L}_p(H^\gamma) - d^{R}_q(H^\gamma).
\end{array}
\end{equation}
\end{Th}

For the proof of Theorm \ref{Th:reduced KKS formula}, see the next Section.

\vskip 0.2cm

Our next result concerns the KKS Fox pairing of regularized holonomies which share at most one tangential endpoint.
In more detail, we consider the path $\gamma_1$ with the starting point $(z_p, 1)$ and the endpoint $(z_q, 1)$, and the path $\gamma_2$ with the starting point $(z_r, 1)$ and the end point $(z_s, 1)$. We assume that all these tangential points are distinct, with a possible exception of $z_r$ and $z_q$ which may (or may not) coincide.

\begin{Prop} \label{th: rho paths}
    The KKS Fox pairing is of the form:

\begin{equation}\label{eq:rho_KKS_holonomy}
\begin{split}
\rho_{\rm{KKS}}\left(H^{\gamma_2},H^{\gamma_1}\right)&=\sum_{a\in\gamma_1\cap\gamma_2}\varepsilon_aH^{\gamma_2}_{z_s\leftarrow a}H^{\gamma_1}_{a\leftarrow z_p}+\delta_{qr}H^{\gamma_2}r_{\text{AM}}(x_r)H^{\gamma_1} \\&-d_s^RH^{\gamma_1}-
d_p^LH^{\gamma_2}+d_q^LH^{\gamma_2}H^{\gamma_1}+H^{\gamma_2}d_r^RH^{\gamma_1}
    \end{split}
    \end{equation}
\end{Prop}

\begin{proof}
In the case of $q=r$, we use the product rule of the reduced coaction and equation \eqref{eq:mubar},
\begin{align*}
 \rho_{\rm{KKS}}(H^{\gamma_2},H^{\gamma_1}) & = \bar{\mu}_\text{KKS}(H^{\gamma_2}H^{\gamma_1}) -  \bar{\mu}_\text{KKS}(H^{\gamma_2})H^{\gamma_1} - H^{\gamma_2} \bar{\mu}_\text{KKS}(H^{\gamma_1}) \\
 & = A^1_{\text{GT}}+B^1_{\zeta}+C^1_{\text{Fox}}
\end{align*}
Here
\begin{align*}
A^1_{\text{GT}} & = \sum_{a\in \gamma_2\gamma_1\cap \gamma_2 \gamma_1 } \varepsilon_a H^{\gamma_2\gamma_1}_{z_s\leftarrow a} H^{\gamma_2\gamma_1}_{a\leftarrow z_p} - \sum_{a\in \gamma_1\cap \gamma_1} \varepsilon_a H^{\gamma_2} H^{\gamma_1}_{z_r\leftarrow a}H^{\gamma_1}_{a\leftarrow z_p} -\sum_{a\in \gamma_2\cap \gamma_2} \varepsilon_a H^{\gamma_2}_{z_s\leftarrow a}H^{\gamma_2}_{a\leftarrow z_r} H^{\gamma_1} \\
& = \sum_{a\in\gamma_1\cap\gamma_2}\varepsilon_aH^{\gamma_2}_{z_s\leftarrow a}H^{\gamma_1}_{a\leftarrow z_p},
\end{align*}
and for a path $\gamma$ we use the notation $\gamma \cap \gamma$ for the set of its self-intersection points. Furthermore,
\begin{align*}
B^1_{\zeta}= & H^{\gamma_2} H^{\gamma_1} r_{\zeta}(-x_p)- r_{\zeta}(x_s) H^{\gamma_2}H^{\gamma_1}+ \text{rot}(\gamma_2\gamma_1)H^{\gamma_2}H^{\gamma_1}-H^{\gamma_2}H^{\gamma_1}r_{\zeta}(-x_p)+ H^{\gamma_2}r_{\zeta}(x_r)H^{\gamma_1} \\
& - \text{rot}(\gamma_1) H^{\gamma_2}H^{\gamma_1} - H^{\gamma_2} r_{\zeta} (-x_r) H^{\gamma_1} + r_{\zeta} (x_s)H^{\gamma_2}H^{\gamma_1}-\text{rot}(\gamma_2)H^{\gamma_2}H^{\gamma_1}\\
=& H^{\gamma_2}r_{\zeta} (x_r)H^{\gamma_1}-H^{\gamma_2}r_{\zeta} (-x_r) H^{\gamma_1}-\frac{1}{2}H^{\gamma_2}H^{\gamma_1}= H^{\gamma_2}r_{\text{AM}}(x_r)H^{\gamma_1}.
\end{align*}
Finally, we have
\begin{align*}
C^1_{\text{Fox}} & =-d^R_s(H^{\gamma_2}H^{\gamma_1})-d^{L}_e(H^{\gamma_2}H^{\gamma_1})+H^{\gamma_2}(d^R_rH^{\gamma_1}+d^{L}_pH^{\gamma_1})+d^R_s(H^{\gamma_2}) H^{\gamma_1}+ d^L_r(H^{\gamma_2})H^{\gamma_1}\\
&=-d_s^RH^{\gamma_1}-
d_p^LH^{\gamma_2}+d_r^LH^{\gamma_2}H^{\gamma_1}+H^{\gamma_2}d_r^RH^{\gamma_1}
\end{align*}
which completes the first part of the proof.

In the case of $q\ne r$, consider a regular path $\gamma$ starting at$(z_q,1)$ and ending at $(z_r,1)$. The product rule implies that
\begin{align*}
\rho_{\text{KKS}}(H^{\gamma_2},H^{\gamma_1}) & = \bar{\mu}_\text{KKS}(H^{\gamma_2}H^{\gamma}H^{\gamma_1}) - H^{\gamma_2} \bar{\mu}_\text{KKS}(H^{\gamma}H^{\gamma_1}) - \bar{\mu}_\text{KKS}(H^{\gamma_2})H^{\gamma}H^{\gamma_1} - \rho_{\rm{KKS}}(H^{\gamma_2},H^{\gamma})H^{\gamma_1}\\
& =A^2_{\text{GT}}+B^2_{\zeta}+C^2_{\text{Fox}},
\end{align*}
Equation \eqref{eq:mubar} the result for $q=r$ allow for an easy computation of the term $A^2_{\text{GT}}$ corresponding to self-intersection points, and of the term $C^2_{\text{Fox}}$ in which one collects the Fox derivatives. We only need to look at the term $B^2_{\zeta}$ in which we collect contributions containing the functions $r(x)$:
\begin{align*}
B^2_{\zeta} & = H^{\gamma_2}H^{\gamma}H^{\gamma_1}r_{\zeta}(-x_p)- r_{\zeta}(x_s) H^{\gamma_2}H^{\gamma} H^{\gamma_1}- H^{\gamma_2}H^{\gamma}H^{\gamma_1}r_{\zeta}(-x_p) +H^{\gamma_2} r_{\zeta} (x_r) H^{\gamma} H^{\gamma_1} \\
&-H^{\gamma_2}r_{\zeta}(-x_r) H^{\gamma} H^{\gamma_1} + r_{\zeta}(x_s) H^{\gamma_2} H^{\gamma} H^{\gamma_1}- H^{\gamma_2}(r_{\zeta}(x_r)- r_{\zeta}(-x_r) -\frac{1}{2}) H^{\gamma}H^{\gamma_1}\\
&+(\text{rot}(\gamma_2 \gamma\gamma_1)-\text{rot}(\gamma \gamma_1)-\text{rot}(\gamma_2))H^{\gamma_2}H^{\gamma}H^{\gamma_1}=0.
\end{align*}
Here we have used the product rule for rotation numbers. This completes the proof.
\end{proof}

The following Proposition gives an explicit (albeit lengthy) formula for the KKS double bracket:

\begin{Prop}
The KKS  double bracket 
of $H^{\gamma_2}$ and $H^{\gamma_1}$ reads
\begin{equation}
\ldb H^{\gamma_2},H^{\gamma_1}\rdb_{\text{KKS}} = A_{\text{GT}} + \delta_{qr} B_{\zeta} + C_{\text{Fox}},
\end{equation}
where
\begin{subequations}
\begin{equation}
A_{\text{GT}}=\sum_{a\in \gamma_1\cap\gamma_2}\varepsilon_a H^{\gamma_1}_{z_q\leftarrow a}H^{\gamma_2}_{a\leftarrow z_r}\otimes H^{\gamma_2}_{z_s\leftarrow a}H^{\gamma_1}_{a\leftarrow z_p}
\end{equation}
\begin{equation}
B_{\zeta}=r_{\text{AM}}(-x_r)'\otimes H^{\gamma_2}r_{\text{AM}}(x_r)^{''}H^{\gamma_1}
\end{equation}
\begin{equation}
\begin{split}
C_{\text{Fox}}=&-H^{\gamma_1}S\left((d_s^RH^{\gamma_1}+d_p^LH^{\gamma_2})'\right)H^{\gamma_2}\otimes (d_s^RH^{\gamma_1}
+d_p^LH^{\gamma_2})''\\
&+S\left((d^{L}_qH^{\gamma_2})'\right)H^{\gamma_2}\otimes (d^{L}_qH^{\gamma_2})''H^{\gamma_1}\\
& + H^{\gamma_1}S\left((d^R_rH^{\gamma_1})'\right)\otimes H^{\gamma_2}(d^R_rH^{\gamma_1})''
\end{split}
\end{equation}
\end{subequations}
\end{Prop}

\begin{proof}
The proof is by a direct calculation using the formula for $\rho_{\text{KKS}}$ in Proposition \ref{th: rho paths}, and the construction of the double bracket in Theorem \ref{th:Fox_to_Double}. 

\end{proof}

Next, we consider the situation a loops $\gamma$ (or several loops) starting and ending at a tangential base point $(z_m, 1)$. 

\begin{Prop}\label{Th:loops}
    For the reduced coaction map of a loop regularized holonomy $H^\gamma$, we have
    \begin{equation}      \label{eq:mubar_loop}
\begin{array}{lll}
 \bar{\mu}_\text{KKS}(H^\gamma)
& =& H^\gamma r_{\zeta}(-x_m) + {\rm rot}(\gamma) H^\gamma  - r_{\zeta}(x_m)H^\gamma
 \\
&+ & \sum_l \varepsilon_l {\rm Hol}^{\rm reg}(\mathcal{A}, \gamma_{[s_l,1]}) {\rm Hol}^{\rm reg}(\mathcal{A}, \gamma_{[0,t_l]}) 
- d^{L}_m(H^\gamma) - d^{R}_m(H^\gamma).
\end{array}
\end{equation}
For a pair of loop $\gamma_1$ and $\gamma_2$ starting and ending at $(z_m, 1)$, we have
\begin{equation}\label{eq:rho_KKS_loop}
\begin{array}{lll}
\rho_{\text{KKS}}(H^{\gamma_2},H^{\gamma_1})
&=&\sum_{a\in \gamma_1\cap \gamma_2}
\varepsilon_a \, H^{\gamma_2}_{z_m\leftarrow a}H^{\gamma_1}_{a\leftarrow z_m}+(H^{\gamma_2}-1) r_{\text{AM}}(x_m)(H^{\gamma_1}-1)\\
&+& \dl_m(H^{\gamma_2})(H^{\gamma_1}-1)+(H^{\gamma_2}-1)\dr_m(H^{\gamma_1})
\end{array}
\end{equation}

\end{Prop}

\begin{proof}
The formula  $\bar{\mu}_\text{KKS}(H^\gamma)$ follows from equation \eqref{eq:mubar} by putting $p=q=m$. 

In order to compute the Fox pairing $\rho_{\text{KKS}}(H^{\gamma_2},H^{\gamma_1})$, we chose two tangential base points $(z_s,1)$ and $(z_e,1)$ such that the three points $z_s,z_m,z_e$ are distinct. We also choose two regular curves $\gamma_0$ from $(z_s,1)$ to $(z_m,1)$ and $\gamma_3$ from $(z_m,1)$ to $(z_e,1)$. 

By the product rule for the Fox pairing and the formula \eqref{eq:rho_KKS_holonomy}, we have the following calculations:
\begin{align*}
\rho_{\text{KKS}}(H^{\gamma_2},H^{\gamma_1}) = & H^{\gamma^{-1}_3}\rho_{\text{KKS}}(H^{\gamma_3}H^{\gamma_2},H^{\gamma_1}H^{\gamma_0})H^{\gamma^{-1}_0}-H^{\gamma^{-1}_3}\rho_{\text{KKS}}(H^{\gamma_3}H^{\gamma_2},H^{\gamma_0})H^{\gamma^{-1}_0} \\
& + H^{\gamma^{-1}_3}\rho_{\text{KKS}}(H^{\gamma_3},H^{\gamma_0})H^{\gamma^{-1}_0}-H^{\gamma^{-1}_3}\rho_{\text{KKS}}(H^{\gamma_3},H^{\gamma_1}H^{\gamma_0})H^{\gamma^{-1}_0} \\
=& A^3_{\text{GT}}+ B^3_{\zeta}+ C^3_{\text{Fox}},
\end{align*}
where
\begin{align*}
A^3_{\text{GT}}=&\sum_{a\in \gamma_3\gamma_2\cap\gamma_1\gamma_0}\varepsilon_a H^{\gamma^{-1}_3}H^{\gamma_3\gamma_2}_{z_e\leftarrow a}H^{\gamma_1\gamma_0}_{a\leftarrow z_s}H^{\gamma^{-1}_0}-\sum_{a\in \gamma_3\gamma_2\cap\gamma_0}\varepsilon_aH^{\gamma^{-1}_3}H^{\gamma_3\gamma_2}_{z_e\leftarrow a}H^{\gamma_0}_{a\leftarrow z_s}H^{\gamma^{-1}_0}\\
&+\sum_{a\in \gamma_3\cap\gamma_0}\varepsilon_a H^{\gamma^{-1}_3}H^{\gamma_3}_{z_e\leftarrow a}H^{\gamma_0}_{a\leftarrow z_s}H^{\gamma^{-1}_0}-\sum_{a\in \gamma_3\cap\gamma_1\gamma_0}\varepsilon_a H^{\gamma^{-1}_3}H^{\gamma_3}_{z_e\leftarrow a}H^{\gamma_1\gamma_0}_{a\leftarrow z_s}H^{\gamma^{-1}_0}\\
=& \sum_{a\in \gamma_1\cap \gamma_2}
\varepsilon_a \, H^{\gamma_2}_{z_m\leftarrow a}H^{\gamma_1}_{a\leftarrow z_m},
\end{align*}

\begin{align*}
B^3_{\zeta}=& H^{\gamma_2}r_{\text{AM}}(x_m)H^{\gamma_1}-H^{\gamma_2}r_{\text{AM}}(x_m)+r_{\text{AM}}(x_m)-r_{\text{AM}}(x_m)H^{\gamma_1}\\
&= (H^{\gamma_2}-1) r_{\text{AM}}(x_m) (H^{\gamma_1}-1),
\end{align*}
and
\begin{align*}
C_{\text{Fox}}=& H^{\gamma^{-1}_3}d^R_e(H^{\gamma_1}H^{\gamma_0})H^{\gamma^{-1}_0} - H^{\gamma^{-1}_3}d^L_s(H^{\gamma_3}H^{\gamma_2})H^{\gamma^{-1}_0}
+H^{\gamma^{-1}_3}d^R_e(H^{\gamma_0})H^{\gamma^{-1}_0}\\
&+H^{\gamma^{-1}_3}d^L_s(H^{\gamma_3}H^{\gamma_2})H^{\gamma^{-1}_0}-H^{\gamma^{-1}_3}d^R_e(H^{\gamma_0})H^{\gamma^{-1}_0}-H^{\gamma^{-1}_3}d^L_s(H^{\gamma_3})H^{\gamma^{-1}_0} \\
&+H^{\gamma^{-1}_3}d^R_e(H^{\gamma_0})H^{\gamma^{-1}_0}+H^{\gamma^{-1}_3}d^L_s(H^{\gamma_3})H^{\gamma^{-1}_0}\\
& + H^{\gamma^{-1}_3}d^L_m(H^{\gamma_3}H^{\gamma_2})H^{\gamma_1}+H^{\gamma_2}d^R_m(H^{\gamma_1}H^{\gamma_0})H^{\gamma^{-1}_0}-H^{\gamma^{-1}_3}d^L_m(H^{\gamma_3}H^{\gamma_2})-H^{\gamma_2}d^R_m(H^{\gamma_0})H^{\gamma^{-1}_0}\\
& + H^{\gamma^{-1}_3}d^L_m(H^{\gamma_3})+ d^R_m(H^{\gamma_0})H^{\gamma^{-1}_0}-H^{\gamma^{-1}_3}d^L_m(H^{\gamma_3})H^{\gamma_1}-d^R_m(H^{\gamma_1}H^{\gamma_0})H^{\gamma^{-1}_0}\\
= & d^L_m(H^{\gamma_2})(H^{\gamma_1}-1)+(H^{\gamma_2}-1)d^R_m(H^{\gamma_1}).
\end{align*}
This completes the proof.

\end{proof}

\begin{Rem}\label{rem: three Fox}
Note that the three maps
$$
\begin{array}{ll}
  \rho_{r_{\rm{AM}}(x_m)}( a , b ) &: =  (a-\varepsilon (a) )r_{\rm{AM}}(x_m) (b-\varepsilon (b) )  \\
  \rho_{L,m} ( a, b ) &: =  d^L_m (a) (b-\varepsilon (b))\\
  \rho_{R,m} ( a, b ) &: = (b- \varepsilon (b)) d^R_m(a)
\end{array}
$$
in the equation \eqref{eq:rho_KKS_loop} are Fox pairings, and in particular $\rho_{r_{\rm{AM}}(x_m)}$ is an inner Fox pairing.
\end{Rem}

\begin{Prop}\label{cor:double bracket loops}
The double bracket of $H^{\gamma_1}$ and $H^{\gamma_2}$ reads
\begin{equation}
\ldb H^{\gamma_2},H^{\gamma_1}\rdb_{\text{KKS}}=A_{\text{GT}}+B_{\zeta}+C_{{\rm Fox}}
\end{equation}
\begin{subequations}
\begin{equation}
A_{GT}=\sum_{a\in \gamma_1\cap\gamma_2} \varepsilon_a H^{\gamma_1}_{z_m\leftarrow a}H^{\gamma_2}_{a\leftarrow z_m}\otimes H^{\gamma_2}_{z_m\leftarrow a}H^{\gamma_1}_{a\leftarrow z_m},\\
\end{equation}
\begin{equation}
B_{\zeta}=(1-H^{\gamma_1})r_{\text{AM}}(-x_m)'(1-H^{\gamma_2})\otimes (H^{\gamma_2^{-1}}-1)r_{\text{AM}}(x_m)''(H^{\gamma_1^{-1}}-1),
\end{equation}
\begin{equation}
\begin{split}
C_{{\rm Fox}}=&(1-H^{\gamma_1}S\left((d^{L}_mH^{\gamma_2})'\right)H^{\gamma_2}\otimes (d^{L}_mH^{\gamma_2})''(H^{\gamma_1}-1),\\
&+H^{\gamma_1}S\left((d^R_mH^{\gamma_1})'\right)(1-H^{\gamma_2})\otimes (H^{\gamma_2}-1)(d^R_mH^{\gamma_1})''.
\end{split}
\end{equation}
\end{subequations}
\end{Prop}

\begin{proof}
We apply Theorem \ref{th:Fox_to_Double}, and the formula for $\rho_{\text{KKS}}$ in Theorem \ref{Th:loops}. 
\end{proof}

The next two Propositions recover expressions for the KKS bracket and cobracket of cyclic quotients of holonomies along closed loops.

\begin{Prop}\label{cor:necklace bracket and cobracket}
For closed loops $\gamma_1,\gamma_2$, we have
\begin{equation}
\begin{split}
\{|H^{\gamma_2}|,|H^{\gamma_1}|\}_{\rm KKS}=&\sum_{a\in \gamma_1\cap \gamma_2}\varepsilon_a|H^{\gamma_1}_{a\leftarrow a}H_{a\leftarrow a}^{\gamma_2}|\\
\end{split}
\end{equation}
\end{Prop}

\begin{proof}
By Proposition \ref{prop: vanish_inner} and Proposition \ref{prop:vanish_partial_inner}, the terms induced by the Fox pairings $\rho_{r_{\rm{AM}}(x_m)} ,\rho_{L,m},\rho_{R,m}$ (see Remark \ref{rem: three Fox}) vanish after taking the cyclic quotient. 
\end{proof}

\begin{Prop}
For a closed loop $\gamma$, we have
\begin{equation}
\begin{split}
\delta_{\rm KKS}(|H^{\gamma}|)=&\sum_{l}\varepsilon_l |\holr(\mathcal{A},\gamma_{[t_l,s_l]})|\wedge |\holr(\mathcal{A},\gamma_{[s_l,1]})\holr(\mathcal{A},\gamma_{[0,t_l]})|+(1\wedge |H|)\rm{rot}(\gamma)\\
\end{split}
\end{equation}
\end{Prop}

\begin{proof}
We apply the formula of the $\delta_{\text{KKS}}$ in Proposition \ref{prop:necklace cobracket} using the formula $\bar{\mu}_{\text{KKS}}(H)$ in Theorem \ref{Th:reduced KKS formula} to obtain
\begin{align*}
\delta_{\text{KKS}}(|H|)=&\sum_{l}\varepsilon_l |\holr(A,\gamma_{[t_l,s_l]})|\wedge |\holr(A,\gamma_{[s_l,1]})\holr(A,\gamma_{[0,t_l]})|+(1\wedge |H|)\rm{rot}(\gamma)\\
&-|H S(d^R_m(H)')|\wedge |d^R_m(H)''|-|HS(d^L_m(H)')|\wedge |d^L_m(H)''|
\end{align*}

From Proposition \ref{prop: left_right_contribution}, we have the relation
\begin{align*}
&|H^\gamma S(d^R_m(H)')|\otimes |d^R_m(H)''|+|HS(d^L_m(H)')|\otimes |d^L_m(H)''|\\
&=|Dd^R_m(H)|+|\alpha\circ\beta^{-1}\circ Dd^R_m(H)|\\
&=|Dd^R_m(H)|+P_{21}|Dd^R_m(H)|
\end{align*}
here $P_{21}$ is the permutation map from $|A|\otimes |A|$ to $|A|\otimes |A|$,
so these terms are symmetric and it vanishes after applying the wedge product.
\end{proof}

Observe that expressions of KKS bracket and cobracket coincide with the Goldman bracket and Turaev cobracket (in the blackboard framing), respectively. This is a manifestation for the Goldman-Turaev formality in genus zero (see \cite{AKNF2018}, also \cite{AF2017}).

\subsection{Representation spaces and applications to Poisson geometry}\label{subsection:representation_space}

We define the space of $N$-dimensional representations of the algebra $A=\K\langle x_1,\dots,x_n\rangle$:
$$
{\rm Rep}(A, N) = {\rm Hom}(A, {\rm Mat}_N(\mathbb{K})).
$$
Since $A$ is a free associative algebra, this representation space is a direct product of $n$ copies of the matrix algebra ${\rm Mat}_N(\mathbb{K})$,
$$
I: {\rm Rep}(A, {\rm Mat}_N(\mathbb{K})) \to {\rm Mat}_N(\mathbb{K}) \times \dots
\times {\rm Mat}_N(\mathbb{K}).
$$
Here the isomorphism $I$ is given by evaluation on generators,
$$
I(\rho) = (\rho(x_1), \dots, \rho(x_n)).
$$
For every $a \in A$, matrix elements $\rho(a)_{ij}$ are functions on the representation space. The standard notation is
$$
a_{ij}(\rho) = \rho(a)_{ij}.
$$
With that we obtain the standard coordinate functions on ${\rm Rep}(A,N)$ as $(x_l)_{ij}$, i.e. the $ij$-th matrix element of the $m$-th matrix.
The following statement is a particular case of the result by Van den Bergh.

\begin{Lemma}[\cite{VandenBergh2008}, Proposition 7.5.2]
Let $\ldb \cdot, \cdot\rdb$ be a double Poisson bracket on $A$. Then,
the representation space ${\rm Rep}(A, N)$ carries a unique Poisson bracket
defined by formula
\begin{equation}\label{formula_representation_space}
\{a_{ij},b_{uv}\}=(\ldb a,b\rdb_{\rm KKS})^{'}_{uj}(\ldb a,b\rdb_{\rm KKS})^{''}_{iv}
\end{equation}
for all $a, b \in A$.
\end{Lemma}

%


The space ${\rm Mat}_N(\K)$ is a Lie algebra, and it can be identified with its dual
$\text{Mat}_N(\K)^*$ by the trace map: $X \mapsto f_X={\rm Tr}(X \cdot)$. Hence, it carries a natural (KKS) Poisson bracket uniquely defined by 
\begin{align*}
\{f_X,f_Y\}(Z):=f_{[X,Y]}(Z) = {\rm Tr}([X,Y]Z),
\end{align*}
for $X, Y, Z \in {\rm Mat}_N(\mathbb{K})$.

\begin{Lemma}
For the KKS double bracket $\ldb \cdot, \cdot\rdb_{\rm KKS}$ on $A$, 
the isomorphism $I$ is a Poisson map under the Van den Bergh Poisson bracket on the representation space  ${\rm Rep}(A, N)$ and the direct product KKS Poisson bracket on the space ${\rm Mat}_N(\mathbb{K})^{\times n}$.
\end{Lemma}

\begin{proof}
It suffices to establish the equality of two Poisson brackets on functions $(x_a)_{ij}(\rho) = \rho(x_a)_{ij}$. First, we compute the direct product Poisson bracket on ${\rm Mat}_N(\mathbb{K})^{\times n}$:
$$
\{ (x_a)_{ij}, (x_b)_{kl}\} = \delta_{ab}(\delta_{jk} (x_a)_{il} - \delta_{il} (x_a)_{kj}).
$$
Next, we compute the Van den Bergh Poisson bracket:
$$
\{ (x_a)_{ij}, (x_b)_{kl}\} = (\ldb x_a, x_b\rdb)'_{kj}(\ldb x_a, x_b\rdb)''_{il} =
\delta_{ab}(\delta_{kj} (x_a)_{il} - \delta_{il} (x_a)_{kj}).
$$
The two results coincide, as required.
\end{proof}

The three not necessarily skew-symmetric Fox-pairings from \ref{rem: three Fox} induce not necessarily skew-symmetric bivector fields on ${\rm Mat}_N(\mathbb{K})^{\times n}$ by the same formula \eqref{formula_representation_space} which are given by the following.
\begin{Lemma}\label{lemma:bivectors for some rhos}
    The bivector fields corresponding to $\rho_{r_{\rm{AM}}(x_m)}$, $\rho_{L,m}$ and $\rho_{R,m}$ are given by
    \begin{align*}
        \Pi_{r_{AM}} &= \sum_{a,b,c,d} r_\text{AM}(\mathrm{ad}_{x_m})_{ab}^{cd} \mathcal{X}^{ab} \otimes \mathcal{X}^{cd} \\
        \Pi_{L,m} &= \sum_{a,b} \frac{\partial}{\partial (x_m)_{ab}} \otimes \mathcal{X}^{ab} \\
        \Pi_{R,m} &=  -\sum_{a,b} \mathcal{X}^{ab} \otimes \frac{\partial}{\partial (x_m)_{ab}}
    \end{align*}
where $\mathcal{X}$ is the diagonal action of $\mathfrak{gl}_N$ given by
\begin{align*}
    \mathcal{X}^{ab} &= \sum_{i,c} \left( (x_i)_{ac} \frac{\partial}{\partial (x_i)_{cb}} - (x_i)_{cb} \frac{\partial}{\partial (x_i)_{ac}} \right)
\end{align*}
\end{Lemma}
\begin{proof}
    For each of the three Fox pairings $\rho$, we have to verify that
    \[
    \left(\ldb x_a, x_b \rdb_{\rho}\right)'_{uj} \left(\ldb x_a, x_b \rdb_{\rho}\right)''_{iv} = \Pi_\rho((x_a)_{ij}, (x_b)_{uv}).
    \]
    We note that since $x_a$ and $x_b$ are primitive, the formula \eqref{eq:pho to double} simplifies to
    \[
    \ldb x_a,x_b \rdb_{\rho} = S(\rho(a,b)') \otimes \rho(a,b)''.
    \]
    Using the identity
    \[
    S(\alpha')_{uj} \alpha_{iv} = \tr\left( \alpha(\ad_{x_1}, \dots, \ad_{x_n})(E_{vu}) E_{ji} \right)
    \]
    we obtain
    \[
     \left(\ldb x_a, x_b \rdb_{\rho}\right)'_{uj} \left(\ldb x_a, x_b \rdb_{\rho}\right)''_{iv} = \tr\left( \rho(x_a,x_b)(\ad_{x_1}, \dots, \ad_{x_n}))(E_{vu}) \cdot E_{ji} \right).
    \]
    In the case $\rho = \rho_{r_{\rm{AM}}}$ we then obtain
    \begin{align*}
    \left(\ldb x_a, x_b \rdb_{\rho}\right)'_{uj} \left(\ldb x_a, x_b \rdb_{\rho}\right)''_{iv} &= \tr\left( \ad_{x_a} r_{\rm AM}(\ad_{x_m}) \ad_{x_b}(E_{vu}) \cdot E_{ji} \right) \\
    &= \sum_{p,q,r,s} r_{\rm AM}(\ad_{x_m})_{pq}^{rs} \,  \mathcal{X}^{pq}\left( (x_b)_{vu} \right)  \mathcal{X}^{rs}\left( (x_a)_{ij} \right),
    \end{align*}
    where we used
    \[
        \ad_{x_b} \left( E_{vu} \right) = \sum_{p,q} E_{pq} \mathcal{X}^{pq}\left( (x_b)_{vu} \right).
    \]
    The other cases $\rho = \rho_{L,m}$ and $\rho = \rho_{R,m}$ work similarly.
\end{proof}

In our Poisson geometry application, we will be using representations of the completed free associative algebra $A=\mathbb{K}\langle\langle x_1, \dots, x_n \rangle\rangle$. As before, a representation $\rho$ of $A$ is completely defined by its values on generators $\rho(x_i)$. However, $\rho(a)$ and functions $a_{ij}(\rho)=\rho(a)_{ij}$ are not defined for $a \in A$ in general. Indeed, we can view $\rho(a)$ as formal power series in matrices $\rho(x_i)$. For $\gamma$ a path or a loop, the regularized holonomy $H^\gamma$ defines a formal power series $\rho(H^\gamma)_{ij}$ with a finite convergence radius for any $N$. Hence, the matrix entries 
$\rho(H^\gamma)_{ij}$ are well defined (transcendental) functions on an open neighborhood of the origin in ${\rm Rep}(A, N)$.

In particular, for a pair of loops $\gamma_1$ and $\gamma_2$ starting and ending at $(z_m,1)$, the elements $H^{\gamma_2},H^{\gamma_1}\in \K\langle\langle x_1,\dots,x_n\rangle\rangle$ define functions $H^{\gamma_2}_{ij}(\rho)=\rho(H^{\gamma_2})_{ij}, H^{\gamma_1}_{uv}(\rho)=\rho(H^{\gamma_1})_{uv}$.
Their Poisson brackets are given by the following Theorem:

\begin{Th}\label{cor:KKS poisson bracketI}
\begin{align*}
\{H^{\gamma_2}_{ij}(x),H^{\gamma_1}_{uv}(x)\}_{\text{KKS}}=& \sum_{a\in \gamma_1\cap\gamma_2} \varepsilon_a \left(H^{\gamma_1}_{z_m\leftarrow a}H^{\gamma_2}_{a\leftarrow z_m}\right)_{uj}\left( H^{\gamma_2}_{z_m\leftarrow a}H^{\gamma_1}_{a\leftarrow z_m}\right)_{iv} \\
&+\Pi(H^{\gamma_2}_{ij}(x),H^{\gamma_1}_{uv}(x))
\end{align*}
where
\begin{align*}
    \Pi &= \sum_{a,b,c,d} r_\text{AM}(\mathrm{ad}_{x_m})_{ab}^{cd} \, \mathcal{X}^{ab} \otimes \mathcal{X}^{cd} + \sum_{a,b} \frac{\partial}{\partial x_m^{ab}} \wedge \mathcal{X}^{ab}
\end{align*}
\end{Th}
\begin{proof}
By Proposition \ref{Th:loops} we obtain
\begin{align*}
\rho_\text{KKS}(H^{\gamma_2}, H^{\gamma_1}) =& \sum_{a\in \gamma_1\cap\gamma_2} \varepsilon_a H^{\gamma_2}_{z_m\leftarrow a}H^{\gamma_1}_{a\leftarrow z_m} \\
&+ \rho_{r_{\mathrm AM}(x_m)}(H^{\gamma_2}, H^{\gamma_1}) \\
&+ \rho_{L,m}(H^{\gamma_2}, H^{\gamma_1}) \\
&+ \rho_{R,m}(H^{\gamma_2}, H^{\gamma_1}).
\end{align*}
We then use formula \eqref{formula_representation_space} together with the calculations in Lemma \ref{lemma:bivectors for some rhos}.
\end{proof}


%
%
%
%
%


\section{Proof of Theorem \ref{Th:reduced KKS formula}}

In this Section we prove Theorem \ref{Th:reduced KKS formula}.
Our main technical tool will be projecting the generalized pentagon equation 
to the  quotient of $U\mathfrak{t}_{n,2}$ by the two-sided ideal $\langle t_{z,w} \rangle^2$. 

\subsection{Operations on the Lie algebra $\mathfrak{t}_{n,2}$}

Consider the Lie ideal $\mathfrak{t}_{n,2} \subset \mathfrak{t}_{n+2}$ defined as the kernel of the Lie homomorphism $\mathfrak{t}_{n+2} \to \mathfrak{t}_n$ that forgets strands $z$ and $w$.

\begin{Prop}\label{prop:presentation of t22}
The Lie algebra $\ft_{n,2}$ has the presentation
\begin{equation} \label{eq:tn2}
\ft_{n,2} = \langle t_{iz}, t_{iw}, t_{zw} \mid [t_{iz}, t_{jw}] = 0 \text{ for $i \neq j$}, [t_{zw}, t_{iz} + t_{iw}] = 0, [t_{iz}, t_{iw} + t_{zw}]=0 \rangle.
\end{equation}
\end{Prop}
\begin{proof}
We have the following split short exact sequence which gives us $\ft_{n+2}\simeq \ft_{n,2}\ltimes \ft_{n}$.
$$
\begin{tikzcd}
    0\arrow{r} & \ft_{n,2}\arrow{r} & \ft_{n+2}\arrow{r}{\text{forget $z,w$}}& \ft_{n}\arrow[bend left=33]{l}{\text{add $z,w$}} \arrow{r} & 0
\end{tikzcd}
$$
Denote the right hand side of \eqref{eq:tn2} by $\mathfrak{h}$. By definition of the $t_{n,2}$, there is a surjective map $\mathfrak{h} \to \ft_{n,2}$. Using the defining relations of $\ft_{n+2}$, one can define 
an action of $\ft_n$  on $\mathfrak{h}$ by derivations. Hence, we have a semi-direct product Lie algebra $\ft_n \ltimes \mathfrak{h}$. Therefore, one can define a Lie algebra homomorphism  $\ft_{n+2}\to \ft_n \ltimes \mathfrak{h}$ given by identity on generators. Since the generators and relations are the same, this map is an isomorphism. In this model of $\ft_{n+2}$, the kernel of the map $\ft_{n+2} \to \ft_{n,2}$ is exactly $\mathfrak{h}$, as required.
\end{proof}

In what follows  we will need the following Lie homomorphisms from the free Lie algebra $\ff_n$ with generators $x_1, \dots x_n$ to $\mathfrak{t}_{n,2}$:
$$
\begin{array}{lll}
\Delta^z_q: x_i \mapsto \begin{cases} t_{iz} + t_{zw} & \text{for $i = q$} \\
                                      t_{iz} & \text{otherwise}
                        \end{cases} \\
\Delta^w_p: x_i \mapsto \begin{cases} t_{iw} + t_{zw} & \text{for $i = p$} \\
                                      t_{iw} & \text{otherwise}
                        \end{cases} \\
\Delta^{zw}: x_i \mapsto t_{iz}+t_{iw} & &
\end{array}
$$

Next, we apply Proposition \ref{prop:forpairing2cocycle} to the Fox pairing $\rho_\text{KKS}$ to obtain an algebra $(A \otimes A \oplus M, \cdot_\rho)$.
\begin{Prop}\label{prop:F}
The assignment
\begin{align*}
    \pi \colon U\ft_{n,2} &\to (A \otimes A \oplus A, \cdot_\rho) \\
    t_{iz} &\mapsto x_i \otimes 1  \\
    t_{iw} &\mapsto 1 \otimes x_i \\
    t_{zw} &\mapsto -e,
\end{align*}
where $e := 0 \oplus 1 \in A \otimes A \oplus M$ extends to an algebra homomorphism. Furthermore, the map $\pi$ factors through the canonical projection $U\ft_{n,2} \to U\ft_{n,2}/\langle t_{zw} \rangle^2$.
\end{Prop}
\begin{proof}
First, we verify that $\pi$ vanishes on relations of $\mathfrak{t}_{n,2}$:
    \begin{align*}
    \pi([t_{iz},t_{jw}]) &= [x_i \otimes 1, 1 \otimes x_j] = x_i \otimes x_j - (x_i \otimes x_j +  \rho(x_j,x_i)e) \\
    &= -\delta_{ij} x_i e = \delta_{ij}\left((x_i \otimes 1) e - e (x_i \otimes 1)\right) \\
    &= \pi(-\delta_{ij}[t_{iz}, t_{zw}]), \\
    \pi([t_{zw},t_{iz} + t_{iw}]) &= [e, x_i \otimes 1 + 1 \otimes x_i] \\
    &= e x_i - x_i e = 0.
    \end{align*}
Next, we observe that $\pi(\langle t_{z,w} \rangle^2) \subset M^2 =0$ since the product of $A \otimes A \oplus M$ vanishes on $M$. This completes the proof.
\end{proof}

By composing $\pi$ with the natural projections 
$A \otimes A \oplus M \to A \otimes A$ and $A \otimes A \oplus M \to  M$ we define the maps
$$
    \pi_0 \colon U\ft_{n,2} \to A \otimes A, \hskip 0.3cm
    \pi_1 \colon U\ft_{n,2} \to M.
$$

\begin{Prop}
    The map $\pi_0$ descends to an algebra isomorphism $U\ft_{n,2}/ \langle t_{zw} \rangle \to A\otimes A$. 
\end{Prop}
\begin{proof}
Since $\pi(\langle t_{z,w} \rangle) \subset M$, the ideal $\langle t_{z,w} \rangle$ is in the kernel of $\pi_0$, and $\pi_0$ descends to a map $U\ft_{n,2}/ \langle t_{zw} \rangle \to A\otimes A$. This map is surjective since the generators $x_i \otimes 1$ and $1 \otimes x_i$ of $A \otimes A$ are in the image. 
It is injective since $\mathfrak{t}_{n,2}/\langle t_{zw} \rangle \cong \mathfrak{f}_n \oplus \mathfrak{f}_n$ with generators $t_{i,z}$ and $t_{i,w}$.
\end{proof}

\begin{Prop} \label{prop:compute pi_1}
    Let $\alpha, \beta \in U\mathfrak{t}_{n,2}$ with either $\pi_0(\alpha) \in A\otimes 1$ or $\pi_0(\beta) \in 1 \otimes A$. Then,
    $$
\pi_1(\alpha \beta) = \pi_0(\alpha)\pi_1(\beta) + \pi_1(\alpha)\pi_0(\beta).
    $$
\end{Prop}
\begin{proof}
    We have that $\pi(\alpha\beta) = \pi(\alpha)\pi(\beta)$. Spelling out the multiplication in $A\otimes A \oplus M$ from \eqref{eqn:multiplicationfromfoxp} we obtain
    \[
    \pi_1(\alpha \beta) = \pi_0(\alpha)\pi_1(\beta) + \pi_1(\alpha)\pi_0(\beta) - \varepsilon(\pi_0'(\alpha)) \rho(\pi_0''(\alpha), \pi_0'(\beta))  \varepsilon(\pi_0''(\beta)),
    \]
    where we denote $\pi_0(\alpha) = \pi_0'(\alpha) \otimes \pi_0''(\alpha)$, and similar for $\pi_0(\beta)$. 
    If $\pi_0(\alpha) \in A \otimes 1$, the last term on the right hand side vanishes (because $\rho(1,-) = 0$). The same argument applies for $\pi_0(\beta) \in 1 \otimes A$.
\end{proof}

We will need the following simple observation:
\begin{Prop}
\begin{equation}
\label{eqn:pi0composedwithmaps}
    \pi_0 \circ \Delta^z_q(\alpha) = \alpha \otimes 1, \hskip 0.3cm
    \pi_0 \circ \Delta_p^w(\alpha) = 1 \otimes \alpha, \hskip 0.3cm
    \pi_0 \circ \Delta^{zw}(\alpha) = \Delta(\alpha),
\end{equation}
where $\Delta$ is the coproduct in the Hopf algebra $A=\K\langle x_1,\dots,x_n\rangle$. 
\end{Prop}

\begin{proof}
The maps on left and right hand sides are algebra homomorphisms. The proof is by checking the equalities on generators.
\end{proof}

\subsection{Applying $\pi$ to pentagon}

First, we compute the images under $\pi_0$ of different terms of the generalized pentagon equation.
We note that $H_z=\Delta^z_q(H), H_w=\Delta^w_p(H)$ and 
$H_{zw}=\Delta^{zw}(H)$.

\begin{Prop}
Let $H = {\rm Hol}^{\rm reg}(A, \gamma)$ and $C_l$ be as in Theorem \ref{thm:generalized_pentagon_equation}. Then, we have
$$
\begin{array}{llll}
\pi_0(\Phi(t_{zw}, t_{wq})) =1, & 
\pi_0(\Phi(t_{pz}, t_{zw})) =1, & 
\pi_0(e^{{\rm rot}(\gamma) t_{zw}})=1, & 
\pi_0(C_l^{-1}e^{-\varepsilon_l t_{z,w}}C_l)=1, \\
\pi_0(\Delta^z_q(H))=H\otimes 1, & 
\pi_0(\Delta^w_p(H))=1\otimes H, &
\pi_0(\Delta^{zw}(H))=H \otimes H. &
\end{array}
$$

\end{Prop}
\begin{proof}
The proof uses equalities \eqref{eqn:pi0composedwithmaps}, 
the fact that $H$ is group-like, and the fact that 
$\Phi(A,B) = \exp(\phi(A,B))$ for some primitive element $\phi(A,B)$ that does not have linear terms in $A$ and $B$.
\end{proof}


Now we can apply $\pi_1$ to the generalized pentagon equation to obtain the following result:

\begin{Prop}\label{prop:step_1}
Let $H = {\rm Hol}^{\rm reg}(A, \gamma)$ and $C_l$ be as in Theorem \ref{thm:generalized_pentagon_equation}. They satisfy the equation
\begin{multline}
\pi_1(\Phi(t_{zw}, t_{wq}))H  
+ \pi_1(\Delta^{zw}(H))
+ {\rm rot}(\gamma) H e
+ H \pi_1(\Phi(t_{pz}, t_{zw})) \\
= 
\pi_1(\Delta^z_q(H)) 
- \sum_l \varepsilon_l \holr(\gamma_{[s_l,1]}) \holr(\gamma_{[0,t_l]}) e 
+ \pi_1(\Delta^w_p(H)).
\label{eqn:pi1appliedtopentagon}
\end{multline}
\end{Prop}
\begin{proof}
We apply the map $\pi_1$ to two sides of the generalized pentagon equation, and use Proposition \ref{prop:compute pi_1}. For the left hand side, we obtain
\begin{align*}
     \pi_1(\Phi(t_{zw}, t_{wq})\Delta^{zw}(H) e^{{\rm rot}(\gamma) t_{zw}} \Phi(t_{pz}, t_{zw})) =& \phantom{+} \pi_1(\Phi(t_{zw}, t_{wq})) (H\otimes H) + \pi_1(\Delta^{zw}(H)) \\
     & + {\rm rot}(\gamma) (H\otimes H)e + (H \otimes H) \pi_1(\Phi(t_{pz}, t_{zw})) \\
    = &\phantom{+} \pi_1(\Phi(t_{zw}, t_{wq}))H + \pi_1(\Delta^{zw}(H)) \\
&+ {\rm rot}(\gamma) H e + H \pi_1(\Phi(t_{pz}, t_{zw})).
\end{align*}
Here we have used the product in $A \otimes A \oplus M$ to obtain
$\pi_1(\Phi(t_{zw}, t_{wq})) (H\otimes H) = \pi_1(\Phi(t_{zw}, t_{wq}))H$ and $(H \otimes H) \pi_1(\Phi(t_{pz}, t_{zw})) =H \pi_1(\Phi(t_{pz}, t_{zw}))$.

For the right hand side, we obtain
\begin{align*}
\pi_1(\Delta^z_q(H) \prod_l(C_l^{-1} e^{-\varepsilon_l t_{zw}} C_l) \Delta^w_p(H))
 = & \phantom{+}
\pi_1(\Delta^z_q(H))(1\otimes H) +  (H\otimes 1) \pi_1(\Delta^w_p(H)) \\
& -  (H\otimes 1)(\sum_l \varepsilon_l \pi_0(C_l)^{-1} e \pi_0(C_l))(1\otimes H) \\
 = & \phantom{+}
\pi_1(\Delta^z_q(H)) +  \pi_1(\Delta^w_p(H)) -  \sum_l \varepsilon_l \pi_0(C_l)^{-1} e \pi_0(C_l)).
\end{align*}
Here we have used Proposition \ref{prop:compute pi_1} and the fact that
$$
\pi_1(C_l^{-1} e^{-\varepsilon_l t_{zw}} C_l) =
\pi_1(C_l^{-1} e^{-\varepsilon_l t_{zw}} C_l-1)=
-\varepsilon_l \pi_0(C_l)^{-1} t_{zw} \pi_0(C_l).
$$
We now use Theorem \ref{thm:generalized_pentagon_equation} to write
\[
\pi_0(C_l) = \holr(\gamma_{[0,t_l]}) \otimes \holr(\gamma_{[1,s_l]})
\]
and to compute
\begin{align*}
\pi_0(C_l)^{-1} e \pi_0(C_l)) =& (\holr(\gamma_{[0,t_l]})^{-1} \otimes \holr(\gamma_{[1,s_l]})^{-1}) e (\holr(\gamma_{[0,t_l]}) \otimes \holr(\gamma_{[1,s_l]})) \\
=& \holr(\gamma_{[1,s_l]})^{-1} e \holr(\gamma_{[0,t_l]}) = \holr(\gamma_{[s_l,1]}) \holr(\gamma_{[0,t_l]}) e.
\end{align*}
\end{proof}

\subsection{The doubling maps} 
Inspired by the generalized pentagon equation, we define the following maps from $A$
to $M$
$$
\square^{zw} =\pi_1 \circ \Delta^{zw}, \hskip 0.3cm
\square^z_q = \pi_1 \circ \Delta^{z}_q, \hskip 0.3cm
\square^w_p = \pi_1 \circ \Delta^{w}_p.
$$
In the following we derive product formulas for these operations. Proposition \ref{prop:compute pi_1} implies that $\pi_1$ is a derivation when restricted to either $\pi_0^{-1}(A \otimes 1)$ or $\pi_0^{-1}(1 \otimes A)$. But by the equations \eqref{eqn:pi0composedwithmaps} the maps $\Delta^z_q$ and $\Delta^w_p$ satisfy this property. So that we readily obtain that $\square^z_q, \square^w_p \colon A \to M$ are derivations where the $A \otimes A$-bimodule is restricted to an $A$-bimodule along $\pi_0 \circ \Delta^z_p$ and $\pi_0 \circ \Delta^w_q$, respectively.

\begin{Prop}\label{prop:step_2}
\label{prop:product fox deri***check}
The maps $\square^z_q$ and $\square^w_p$ are right and left Fox derivatives, respectively. Moreover,
$$
\square^z_q =d^{R}_q, \hskip 0.3cm
\square^w_p =d^{L}_p.
$$
\end{Prop}
\begin{proof}
The maps $\square^z_q$ and $\square^w_p$ are derivations. Using the description of $\pi_0 \circ \Delta^z_q$ and $\pi_0 \circ \Delta^w_p$ we obtain
$$
\begin{array}{lll}
\square^z_q(fg) & = & \square^z_q(f) (g\otimes 1) + (f\otimes 1) \square^z_q(g), \\
\square^w_p(fg) & = & \square^w_p(f) (1\otimes g) + (1\otimes f) \square^w_p(g).
\end{array}
$$
The $A\otimes A$-bimodule structure on $M$ implies the Fox derivative property. To identify the Fox derivatives, we check on generators
\begin{equation*}
\square^{z}_q(x_j)=\delta_{qj}=d^{R}_q(x_j),\quad \square^{w}_p(x_j)=\delta_{pj}=d^{L}_p(x_j)
\end{equation*}
\end{proof}

We now turn to the map $\square^{zw}$.
\begin{Prop}
    The map $\square^{zw}$ is a reduced coaction with respect to $-\rho$, that is
    \[
    \square^{zw}(uv) = u \square^{zw}(v) + \square^{zw}(u) v - \rho(u, v) 
    \]
\end{Prop}
\begin{proof}
We use that $\pi \circ \Delta^{zw}$ is an algebra homomorphism and the description of the product to obtain
\[
\pi_1(\alpha \beta) = \pi_0(\alpha)\pi_1(\beta) + \pi_1(\alpha)\pi_0(\beta) - \varepsilon(\pi_0'(\alpha)) \rho(\pi_0''(\alpha), \pi_0'(\beta))  \varepsilon(\pi_0''(\beta))
\]
for $\alpha = \Delta^{zw}(u)$ and $\beta = \Delta^{zw}(v)$. Using that $\pi_0 \circ \Delta^{zw} = \Delta$ we obtain
\[
\square^{zw}(uv) = \Delta(u) \square^{zw}(v) + \square^{zw}(u) \Delta(v) - \rho(u, v) 
\]
By using the description of the $A \otimes A$-bimodule structure on $M$, the expression above simplifies to the required formula.    
\end{proof}

\begin{Prop}\label{prop:reduced KKS}\label{prop:step 4}
    We have $\overline{\mu}_\text{KKS} =-\square^{zw}$.
\end{Prop}
\begin{proof}
    Recall that $\overline{\mu}_\text{KKS}$ is a reduced coaction with respect to $\rho_{\text{KKS}}$, and $\square^{zw}$ is a reduced coaction with respect to $-\rho_{\text{KKS}}$.
    Then, it suffices to check their action on generators. We compute
    \[
    \pi( \Delta^{zw})(x_i) = \pi(t_{iz} + t_{iw}) = x_i \otimes 1 + 1 \otimes x_i.
    \]
    In particular, we find $\pi_1(\Delta^{zw}(x_i)) = \square^{zw}(x_i)= 0 = \overline{\mu}_\text{KKS}(x_i)$, as required.
\end{proof}

\subsection{Associator terms and conclusion}
In the following Proposition we compute the $\pi_1$ projection of the associator terms of the generalized pentaton equation.

\begin{Prop}\label{prop:step_5}
We have
\begin{equation}
\pi_1(\Phi(t_{zw},t_{wq})) =-r_{\zeta}(x_q), \hskip 0.3cm
\pi_1(\Phi(t_{pz},t_{zw}))  =r_{\zeta}(-x_p).   
\end{equation}
\end{Prop}
\begin{proof}
    We recall the following asymptotics of the KZ-associator $\Phi$ (see \cite{LeMurakami1996}): 
    \begin{equation*}
    \begin{split}
    &\Phi(2 \pi i A, 2 \pi i B)=1+ \sum_{m=2}^{\infty}(-1)\zeta(m)(\text{ad}_A)^{m-1}B+O(B^2).
    \end{split}
    \end{equation*}
    This implies
    \begin{align*}
    \pi(\Phi(t_{pz},t_{zw})) &= \Phi(x_p \otimes 1, e) = 1 + \frac{1}{2 \pi i}\sum_{m=2}^{\infty}(-1)\zeta(m)(\text{ad}_{\frac{x_p \otimes 1}{2\pi i} })^{m-1} e \\
    &= 1 + \frac{1}{2 \pi i} \sum_{m=2}^{\infty}(-1)\zeta(m) e\left(\frac{-x_p}{2 \pi i}\right)^{m-1} = 1 + e \, r_\zeta(-x_p)
    \end{align*}
    where we have used that $\text{ad}_{x_p \otimes 1}(m) = -mx_p$ for any $m \in M$. This implies $\pi_1(\Phi(t_{pz},t_{zw})=r_\zeta(-x_p)$, as required.
    
    A similar calculation yields
    \[
    \pi(\Phi(t_{pw}, t_{zw})) =  1 + r_\zeta(x_p) e.
    \]
    We now use that $\Phi(t_{zw}, t_{wp})\Phi(t_{wp}, t_{zw})) = 1$ to obtain
    $$
    1 = \pi(\Phi(t_{zw}, t_{wp})) \pi(\Phi(t_{wp}, t_{zw})) 
     = (1 + \pi_1(\Phi(t_{wp}, t_{zw})))(1 + r_\zeta(x_p) e)
    $$
    which implies $\pi_1(\Phi(t_{wp}, t_{zw}))= -  r_\zeta(x_p)$.
\end{proof}

Putting together Propositions \ref{prop:step_1},~\ref{prop:step_2},~\ref{prop:step 4},~\ref{prop:step_5}, we are now ready to evaluate the LHS of equation \eqref{eqn:pi1appliedtopentagon}:
\begin{multline}
\pi_1(\Phi(t_{zw}, t_{wq}))H  
+ \pi_1(\Delta^{zw}(H))
+ {\rm rot}(\gamma) H 
+ H \pi_1(\Phi(t_{pz}, t_{zw})) \\
= 
-r_{\zeta}(x_q)H-\bar{\mu}_{\text{KKS}}(H) + \text{rot}(\gamma)H + Hr_{\zeta}(-x_p).
\end{multline}
The RHS of equation \eqref{eqn:pi1appliedtopentagon} reads
\begin{multline}
\pi_1(\Delta^z_q(H)) 
- \sum_l \varepsilon_l \holr(\gamma_{[s_l,1]}) \holr(\gamma_{[0,t_l]})  
+ \pi_1(\Delta^w_p(H))\\
=d^{R}_q(H) - \sum_l \varepsilon_l \holr(\gamma_{[s_l,1]}) \holr(\gamma_{[0,t_l]})  + d^{L}_p(H).
\end{multline}
Equating the LHS and the RHS we obtain
\begin{equation*} 
\begin{array}{lll}
 \bar{\mu}_\text{KKS}(H)
& =& H r_{\zeta}(-x_p) + {\rm rot}(\gamma) H  - r_{\zeta}(x_q)H
 \\
&+ & \sum_l \varepsilon_l {\rm Hol}^{\rm reg}(\mathcal{A}, \gamma_{[s_l,1]}) {\rm Hol}^{\rm reg}(\mathcal{A}, \gamma_{[0,t_l]}) 
- d^{L}_p(H) - d^{R}_q(H),
\end{array}
\end{equation*}
as required.

\end{document}